\documentclass[12pt, reqno]{amsart}

\usepackage{enumerate, amsmath, amsthm, amsfonts, amssymb, xy,  mathrsfs, graphicx, paralist, fancyvrb,mathtools}
\usepackage[usenames, dvipsnames]{xcolor}
\usepackage[margin=1in]{geometry} 
\usepackage[bookmarks, colorlinks=true, linkcolor=blue, citecolor=blue, urlcolor=blue]{hyperref}

\usepackage{lscape}

\input xy
\xyoption{all}

\usepackage{tikz,tikz-cd}	
\usetikzlibrary{positioning, matrix, shapes} 
\usepackage[enableskew]{youngtab}

\usepackage{titlesec}		
\titleformat{\section}[block]{\large\scshape\bfseries\filcenter}{\thesection.}{1em}{}		
\titleformat{\subsection}[hang]{\large\scshape\bfseries}{\thesubsection}{1em}{}			
\titleformat{\subsubsection}[hang]{\large\scshape\bfseries}{\thesubsubsection}{1em}{}			

\usepackage[titles]{tocloft}								     					
\usepackage{multirow}
\usepackage{array}
\usepackage{booktabs}
\newcolumntype{M}[1]{>{\centering\arraybackslash}m{#1}}
\newcolumntype{N}{@{}m{0pt}@{}}
\usepackage{diagbox}
\usepackage{cancel}

\numberwithin{equation}{subsection}
\newtheorem{theorem}[equation]{Theorem}

\newtheorem{proposition}[equation]{Proposition}
\newtheorem{lemma}[equation]{Lemma}
\newtheorem{corollary}[equation]{Corollary}

\theoremstyle{definition}
\newtheorem{rmk}[equation]{Remark}
\newenvironment{remark}[1][]{\begin{rmk}[#1] \pushQED{\qed}}{\popQED \end{rmk}}
\newtheorem{eg}[equation]{Example}
\newenvironment{example}[1][]{\begin{eg}[#1] \pushQED{\qed}}{\popQED \end{eg}}
\newtheorem{defn}[equation]{Definition}
\newenvironment{definition}[1][]{\begin{defn}[#1]\pushQED{\qed}}{\popQED \end{defn}}

\makeatletter
\@addtoreset{equation}{section}
\renewcommand{\thesubsection}{%
  \ifnum\c@subsection<1 \@arabic\c@section
  \else \thesection.\@arabic\c@subsection
  \fi
}
\makeatother

\newcommand{\cB}{\mathcal{B}}

\newcommand{\bC}{\mathbf{C}}
\newcommand{\cC}{\mathcal{C}}
\newcommand{\fC}{\mathfrak{C}}

\newcommand{\bF}{\mathbf{F}}

\newcommand{\cG}{\mathcal{G}}

\newcommand{\cH}{\mathcal{H}}

\newcommand{\cJ}{\mathcal{J}}

\newcommand{\bK}{\mathbf{K}}
\newcommand{\cK}{\mathcal{K}}

\newcommand{\bN}{\mathbf{N}}

\newcommand{\bP}{\mathbf{P}}

\newcommand{\bQ}{\mathbf{Q}}

\newcommand{\cS}{\mathcal{S}}
\newcommand{\fS}{\mathfrak{S}}

\newcommand{\fT}{\mathfrak{T}}

\newcommand{\bU}{\mathbf{U}}

\newcommand{\bV}{\mathbf{V}}
\newcommand{\cV}{\mathcal{V}}

\newcommand{\bZ}{\mathbf{Z}}

\newcommand{\bh}{\mathbf{h}}

\newcommand{\bk}{\mathbf{k}}

\newcommand{\bx}{\mathbf{x}}

\renewcommand{\phi}{\varphi}
\renewcommand{\emptyset}{\varnothing}

\newcommand{\injects}{\hookrightarrow}

\renewcommand{\tilde}[1]{\widetilde{#1}}
\newcommand{\ol}[1]{\overline{#1}}
\newcommand{\ul}[1]{\underline{#1}}

\makeatletter
\def\Ddots{\mathinner{\mkern1mu\raise\p@
\vbox{\kern7\p@\hbox{.}}\mkern2mu
\raise4\p@\hbox{.}\mkern2mu\raise7\p@\hbox{.}\mkern1mu}}
\makeatother

\DeclareMathOperator{\im}{image} 

\DeclareMathOperator{\coker}{coker}
\renewcommand{\hom}{\operatorname{Hom}}

\newcommand{\GL}{\mathbf{GL}}

\DeclareRobustCommand{\gobblefour}[5]{}

\begin{document}

\title{Representation Stability for Sequences of $0$-Hecke Modules}
\author{Robert P. Laudone}
\address{Department of Mathematics, University of Wisconsin, Madison}
\email{laudone@wisc.edu \newline \indent {\em URL:} \url{https://www.math.wisc.edu/~laudone/}}
\date{\today}
\thanks{RL was supported by NSF grant DMS-1502553.}

\maketitle

\begin{abstract}
We define a new category analogous to ${\bf FI}$ for the $0$-Hecke algebra $H_n(0)$ called the $0$-Hecke category, $\cH$, indexing sequences of representations of $H_n(0)$ as $n$ varies under suitable compatibility conditions. We establish a new type of representation stability in this setting and prove it is implied by being a finitely generated $\cH$-module. We then provide examples of $\cH$-modules and discuss further desirable properties these modules possess.
\end{abstract}

{
  \hypersetup{linkcolor=black}
  \tableofcontents
}

\section{Introduction}
The category {\bf FI} of finite sets and injections, first defined in \cite{CEF}, and its variants have been of great interest recently. Being a finitely generated ${\bf FI}$-module implies many desirable properties that are often very difficult to prove on their own about sequences of symmetric group representations, such as representation stability and polynomial growth. The study of this combinatorial category and its modules has been fruitful, providing tools to prove a variety of stability results about spaces such as $H^i({\rm Conf}_n(M); \bQ)$ the cohomology of configuration space of $n$ distinct ordered points on a connected, oriented manifold $M$ and many others \cite{CEF}. In this vein, a variety of other combinatorial categories have been defined. The most relevant categories to us that have seen a surge of interest in recent years are ${\bf FI}$ and ${\bf OI}$, for a general survey of results we refer the reader to \cite{SS2, SS1, SS4, CEF}. Recently, in \cite{GS} the authors establish a structure theory for ${\bf OI}$-modules similar to the structure theory for ${\bf FI}$-modules in \cite{SS4} by studying $\cJ$ the monoid of increasing functions.

One can loosely think of ${\bf FI}$-modules as sequences of compatible spaces with a complete symmetry present, and similarly ${\bf OI}$-modules can be thought of as sequences of compatible spaces with no requirement of symmetry. A similar story plays out in the analogous field of symmetric function theory as seen in \cite{HLMW}, where the ring of symmetric functions ${\bf Sym}$ plays the role of ${\bf FI}$ and the ring of nonsymmetric functions plays the role of ${\bf OI}$. Interpolating between these two rings is the ring of quasisymmetric functions ${\bf QSym}$. One of the goals of this paper is to define the analogous category that interpolates between ${\bf FI}$ and ${\bf OI}$ and explore its surprising structure. 

This category turns out to be the categorical analogue of ${\bf FI}$ for the $0$-Hecke algebra. We denote it by $\cH$ and call it the {\em $0$-Hecke category}. We begin by defining $\cH$ and discussing how it can be viewed as a quotient of the braid category as seen in \cite[\S 1.2]{WR}. We then prove a variety of structural results about $\cH$-modules. Most notably, despite the $0$-Hecke algebra not being semi-simple, one can say a surprising amount about the underlying $H_n(0)$ representation theory of a finitely generated $\cH$-module. In particular, we are able to construct an explicit basis for the Grothendieck group of $\cH$-modules $\cG({\rm Mod}_\cH)$ consisting of {\em padded induced modules} $M(\alpha, k)$ where $\alpha$ is a composition and $-1 \leq k \leq |\alpha|$ is an integer. These modules are quotients of the standard induced modules. This ultimately allows us to prove that finitely generated $\cH$-modules satisfy a new form of representation stability,

\begin{theorem}\label{Theorem A}
For any finitely generated $\cH$-module $V$, we have a unique finite decomposition in the Grothendieck group $\cG({\rm Mod}_\cH)$,
\[
[V_n] = \sum_{i,j} c_{\alpha_i,k_j} [M(\alpha_i,k_j)_n]
\]
where the coefficients $c_{\alpha_i,k_j}$ do not depend on $n$.
\end{theorem}

As with the standard form of representation stability for ${\bf FI}$, this theorem implies that a finite list of data, namely $\{(\alpha_i,k_i)\}$, completely describe all the irreducible $H_n(0)$ representations that will occur in $V_n$ for $n \gg 0$. We also describe the exact process to go from the finite list of data to the irreducible representations one desires. 

We also study an important statistic often associated with these combinatorial categoriesl Gabriel-Krull dimension due to \cite{Gei}. One can think of this as the analogue of Krull dimension in commutative algebra, it gives a rough measure of the complexity of the category. We prove,

\begin{theorem}
The Gabriel-Krull dimension of $\cH$ is infinite.
\end{theorem}

In this respect, $\cH$ bears more similarity to ${\bf OI}$ which \cite{GS} shows has infinite Gabriel-Krull dimension. We dedicate the remainder of the paper to exploring examples of $\cH$-modules. We expect there are many more examples,
\begin{enumerate}
\item In Section \ref{polynomialSec} we explore one of the more basic but tractable finitely $\cH$-modules $[n] \mapsto k[x_1,\dots,x_n]_d$ the degree $d$ polynomials in $n$ variables. It is not hard to show this forms a $\cH$-module where $H_n(0)$ acts by Demazure operators. This example is instructive because for small values of $d$ we can explicitly see the representation stability predicted by Theorem \ref{Theorem A}.
\item In Section \ref{CohomologySection} we present one of the motivating examples for this paper. From work of \cite{RW} and \cite{HR} it was known that the group cohomology $H^i(B(n,q),\bF_q)$ where $B(n,q)$ is the Borel subgroup of $\GL(n,q)$ carried an action of $H_n(0)$ for any fixed $i \geq 0$. It does not, however, have any natural complete symmetry, i.e. an action of $\fS_n$. In this setting it is natural to ask if a $\cH$-module structure is present, we are able to prove

\begin{theorem} \label{Theorem B}
The assignment $[n] \mapsto H^i(B(n,q),\bF_q)$ with compatibility maps described in Section \ref{CohomologySection} forms a $\cH$-module.
\end{theorem}

We believe this defines a finitely generated $\cH$-module, but this problem is still open.

\item In Section \ref{HomologySection} we study one of the other motivating examples for this paper, $H_i(B(n,q),\bF_q)$. By a spectral sequence argument $H_i(B(n,q),\bF_q)$ is equal to $H_i(U(n,q),\bF_q)$, where $U(n,q)$ is the unipotent subgroup. This motivated Putman, Sam and Snowden to prove that the assignment $[n] \mapsto H_i(U(n,q),\bF_q)$ with inclusion maps induced by inclusion of Unipotent subgroups is a finitely generated ${\bf OI}$-module \cite{PSS}. We define a new $H_n(0)$-action on $H_i(U(n,q),\bF_q)$ and prove

\begin{theorem} \label{Theorem B2}
The assignment $[n] \mapsto H_i(U(n,q),\bF_q)$ with compatibility maps 
\[
\Phi_n \colon H_i(U(n,q),\bF_q) \to H_i(U(n+1,q),\bF_q)
\]
induced by the natural inclusion $U(n,q) \injects U(n+1,q)$ is a finitely generated $\cH$-module.
\end{theorem}

It turns out that the restriction to ${\bf OI}$ of the $\cH$-action gives an action equivalent to the one defined in \cite{PSS}. This result is a strengthening of their result because $\cH$ has more structure. In particular, Theorem \ref{Theorem B2} implies

\begin{theorem} \label{repStable}
For any fixed $i \geq 0$, the sequence of $H_n(0)$-modules $\{H_i(U(n,q),\bF_q)\}_{n \geq 0}$ is representation stable.
\end{theorem}
\item In Section \ref{StanleySec} we explain how we can turn the collection of Stanley Reisner rings of the Boolean algebra into a $\cH$-module using work of \cite{Hu}. We then prove this $\cH$-module is finitely generated and so is representation stable in our new sense.

\item Finally in Section \ref{QuasiSchurSec} we make the connection to ${\bf QSym}$ more explicit and explore how a new class of modules defined by \cite{TW} which map to quasisymmetric Schur functions as defined in \cite{HLMW} under the quasisymmetric characteristic map have a finitely generated $\cH$-module structure. These modules provide another example of a $\cH$-module that cannot be a ${\bf FI}$-module.
\end{enumerate}

\subsection{Outline of Argument}
The proof of Theorem \ref{Theorem A} breaks into the following steps
\begin{enumerate}
\item For each composition $\alpha$, we first define an important class of $\cH$-modules called induced modules $M(\alpha)$ in \S\ref{InducedSec} and use this to define padded induced modules $M(\alpha,k)$ in \S\ref{PaddedInducedSec}.
\item Every finitely generated $\cH$-module has eventually polynomial growth, the degree of the polynomial that eventually describes the growth is called the {\em polynomial degree} of the $\cH$-module. We prove that padded induced modules have the property that any proper quotient of them has strictly smaller polynomial degree in Theorem \ref{quotients}.
\item We then argue in Theorem \ref{smallest} that $M(\alpha,k)$ is the smallest polynomial degree $k$ quotient of $M(\alpha)$, i.e. any other polynomial degree $k$ quotient of $M(\alpha)$ must contain $M(\alpha,k)$.
\item The above allows us to deduce that $M(\alpha,k)$ is simple in the Serre quotient $\cH_{k}$ of polynomial degree $\leq k$ objects by polynomial degree $\leq k-1$ objects and that every simple is of this form. This can be seen in Lemma \ref{simple}.
\item Since the $M(\alpha) = M(\alpha,|\alpha|)$ are simple and every finitely generated $\cH$-module is a quotient of a direct sum of the $M(\alpha)$, we can then argue that every object in the Serre quotient has finite length (Lemma \ref{finite length}). This allows us to find a finite length filtration of any finitely generated polynomial degree $k$ $\cH$-module $V$ in $\cH_{k}$ with successive quotients isomorphic to $M(\alpha_i,k)$ (Theorem \ref{decomp}).
\item We then use this filtration to show that the isomorphism classes $[M(\alpha,k)]$ as $\alpha$ ranges over all compositions and $k \in \bZ$ ranges $-1 \leq k \leq |\alpha|$ span the Grothendieck group of finitely generated $\cH$-modules, $\cG({\rm Mod}_\cH)$ in Theorem \ref{span}.
\item Finally, we construct functions on the Grothendieck group to argue that the $[M(\alpha,k)]$ are in fact linearly independent and hence form a basis for $\cG({\rm Mod}_\cH)$ (Theorem \ref{GrothBasis}). This ultimately implies our new form of representation stability.
\end{enumerate}

\subsection{Relation to Previous Work}
\begin{itemize}
\item We show in Theorem \ref{HNoeth} that $\cH$ is another example of an ever growing class of noetherian categories. In \cite{SS1} the authors describe a Gr\"obner method for proving noetherianity of combinatorial categories which we apply here. This underlying idea fits into a broader area of interest called {\em noetherianity up to symmetry}. For a nice introduction we recommend \cite{Dra}. Ultimately, one works with a space or object on which a group or algebra acts and proves finite generation up to the action of this group or algebra. Noetherianity up to symmetry is important in \cite{SS2, SS3, SS4, NSS}, where the authors explore various manifestations of this idea to prove finite generation results for various representations of categories and twisted commutative algebras. These ideas are also present in \cite{CEF, Sn, Lau, DE, To, Sa1} and many other recent papers.
\item As we have mentioned one can view $\cH$ as interpolating between ${\bf FI}$ and ${\bf OI}$. This paper opens the door to many questions explored in \cite{CEF} for {\bf FI}-modules and \cite{SS1, GS} for ${\bf OI}$-modules.
\item $\cH$ has concrete connections to ${\bf QSym}$, the ring of quasisymmetric functions, as seen in \S\ref{QuasiSchurSec}. This ring has seen a resurgence of interest as of late, for a general survey we refer the reader to \cite{Mas}. The connection between ${\bf FI}$ and the ring of symmetric functions is very explicit, in particular the authors in \cite{SS4} prove that the Grothendieck group of ${\bf FI}$-modules is isomorphic to two copies of the ring of symmetric functions. It would be interesting to see if a similar result were true about $\cH$ and what else the category could tell us about ${\bf QSym}$. 
\item This paper also provides a more systematic approach to studying natural sequences of representations of the $0$-Hecke algebra. Recent such examples addressed in this paper include the Stanley Reisner ring of the Boolean algebra studied in \cite{Hu}, quasisymmetric Schur modules studied in \cite{TW}, and $H^i(B(n,q),\bF_q)$ as studied in \cite{RW, HR}. There are many other examples we are working on in a future paper including ordered set partitions studied in \cite{HuRh}. 
\item Quillen computed the cohomology of the general linear group in non-defining characteristic, $H^i(\GL(n,q),\bF_\ell)$, using Sylow subgroups as a key ingredient \cite{Qui}. In defining characteristic, $B(n,q)$ is a Sylow-$q$ subgroup of $\GL(n,q)$, and so researchers have naturally sought to understand $H^i(B(n,q); \bF_q)$. Quillen's methods depend on working over non-defining characteristic and the pre-existing ${\bf FI}$ literature does not naturally apply because there is no natural symmetric group action. The category $\cH$ provides the missing categorical framework to study the group homology of the Borel group in defining characteristic.
\end{itemize}

\subsection{Conventions}
For the majority of the paper $\bk$ denotes a field of arbitrary characteristic because the representation theory of $H_n(0)$ does not depend on characteristic. The majority of the notation for this paper is outlined in \S \ref{NotationSec}.

\ \\ \
\noindent {\large {\bf Acknowledgements.}} I thank Steven Sam for his constant guidance and helpful conversations. I would also like to thank Brendon Rhoades for his openness to discuss this paper with me.

\section{Background} \label{NotationSec}
\subsection{Compositions}
A {\em composition} $\alpha$ of $n$ is a list of non-negative natural numbers $(\alpha_1,\dots,\alpha_k)$ such that $\alpha_1 + \cdots + \alpha_k = n$. There are $2^{n-1}$ compositions of $n$. We will denote that $\alpha$ is a composition of $n$ by writing $\alpha \vDash n$. Compositions of $n$ isomorphically correspond to {\em descent sets} on $[n-1]$, subsets of increasing integers in $[n-1]$, in the following way. Given a composition $\alpha = (\alpha_1,\dots,\alpha_k)$ we write $D(\alpha) = \{\alpha_1,\alpha_1+ \alpha_2,\dots,\alpha_1+ \cdots + \alpha_{k-1}\}$. It is also easy to reverse this process. Given a descent set $D$, we denote by $C(D)$ the corresponding composition.

We can view compositions combinatorially as ribbon diagrams. A {\em ribbon diagram} is an edgewise connected skew diagram with no $2$ by $2$ boxes. The composition $\alpha \vDash n$ given by $(\alpha_1,\dots,\alpha_k)$ corresponds to the ribbon diagram with $\alpha_i$ boxes in row $i$. For example, $(2,1,3)$ corresponds to the ribbon diagram
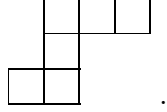
\[
\young(:~~~,:~::,~~::) \;.
\]

Given a composition $\alpha$, let $\ell(\alpha)$, the {\em length} of $\alpha$, denote the number of entries in $\alpha$. Combinatorially, this is the number of rows in the ribbon diagram of shape $\alpha$. We denote by $w(\alpha)$ the {\em width} of $\alpha$. Combinatorially, this gives the number of columns in the ribbon diagram of $\alpha$. 

\begin{proposition}
For any composition $\alpha \vDash n$, $w(\alpha) = n - \ell(\alpha) + 1$.
\end{proposition}

\begin{proof} 
To see why this is the case, notice that the number of columns in the ribbon diagram is exactly,
\[
(\alpha_1) + (\alpha_2-1) + \cdots + (\alpha_k-1) = (\alpha_1 + \cdots + \alpha_k) - k + 1 = n - \ell(\alpha) + 1. \qedhere
\]
\end{proof}

Given two compositions $\alpha,\beta$ we define their {\em sum} to be $\alpha + \beta = (\alpha_1,\dots,\alpha_k,\beta_1,\dots,\beta_\ell)$ and their {\em join} as $\alpha \lhd \beta = (\alpha_1,\dots,\alpha_k+\beta_1,\beta_2,\dots,\beta_\ell)$.

We will often discuss adding boxes to positions of a ribbon diagram, we make this concrete here. Given a composition $\alpha = (\alpha_1,\dots,\alpha_k)$ with $\alpha \vDash n$, adding a box to the $i$th {\em position} of $\alpha$ means that we add a box above the $i-1$st box in the ribbon diagram corresponding to $\alpha$ and shift the corresponding boxes up to ensure we still have a ribbon diagram. Explicitly if the $i-1$st box lives in $\alpha_j$, i.e. $\alpha_1 + \cdots + \alpha_{j-1} < i-1 \leq \alpha_1 + \dots + \alpha_j$, then the ribbon diagram where we add a box to position $i-1$ is exactly
\[
(\alpha_1,\dots,\alpha_{j-1},\alpha_j-\beta_i,1) \lhd (\beta_i,\alpha_{j+1},\dots,\alpha_k),
\]
where $\beta_i$ is a correction factor so that $\alpha_1+\cdots+ \alpha_j - \beta_i = i-1$. 

\begin{example}
Say we wish to add a box to the second position of the ribbon diagram corresponding to $(2,1,3)$. In terms of our explicit description we see that the first box lies in $\alpha_1$ and we need $\beta_2 = 1$ so that $2 - \beta_2 = 1$. So the diagram where we add a box in the second position should be
\[
(1,1) \lhd (1,1,3) = (1,2,1,3).
\]
Expressed as a ribbon diagram,
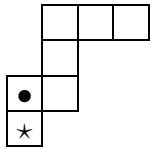
\[
\young(:~~~,:~:::,\bullet~:::,\star::::)
\]
where we denote by ${\tiny \young(\bullet)}$ the new box we added. Combinatorially, this is exactly the same as adding a box above the box ${\tiny \young(\star)}$ in the diagram
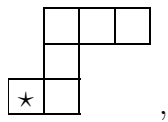
\[
\young(:~~~,:~::,\star~::) \;,
\]
where we then shift the remaining boxes up to ensure we still have a ribbon diagram.
\end{example} 
From this we can see that in a composition of $n$, it makes sense to add a box to any position $i$ with $1 \leq i \leq n+1$. We say that $\alpha \vDash n$ has $n+1$ positions, so ${\rm position}(\alpha) = |\alpha|+1$. Combinatorially this is because we are allowed to add a box to the left of the first box as well.

\subsection{$0$-Hecke Algebra}
The {\bf $0$-Hecke algebra} $H_n(0)$ is the associative algebra generated by $\pi_1,\dots,\pi_{n-1}$ over an arbitrary field $\bF$ where the generators satisfy the relations
\[
\begin{cases}
\pi_i^2 = \pi_i &1 \leq i \leq n-1\\
\pi_i \pi_{i+1} \pi_i = \pi_{i+1} \pi_i \pi_{i+1} &1 \leq i \leq n-2\\
\pi_i \pi_j = \pi_j \pi_i &|i-j| > 1.
\end{cases}
\]

We will most often use this generating set. The last two relations are commonly called the {\em braid relations} and the first relation is sometimes called the {\em skein relation}. $H_n(0)$ is a deformation of the symmetric group $\fS_n$ and one of the most well studied degenerate deformations. Recall that the symmetric group has Coxeter generators $\{s_1,\dots,s_{n-1}\}$ where the $s_i = (i,i+1)$ is the adjacent transposition. These generators satisfy the braid relations, but not the skein relation, instead $s_i^2 = 1$. If $w \in \fS_n$ is a permutation and $w = s_{i_1} \cdots s_{i_k}$ is a reduced expression for $w$ in the Coxeter generators we define the $H_n(0)$ algebra element $\pi_w \coloneqq \pi_{i_1} \cdots \pi_{i_k} \in H_n(0)$. One can show that the set $\{\pi_w \; | \; w \in \fS_n\}$ forms a basis for $H_n(0)$ as an $\bF$-vector space. 

 It is well known that $H_n(0)$ has another algebra generating set $\{\ol{\pi}_1,\dots,\ol{\pi}_{n-1}\}$ subject to the relations
\[
\begin{cases}
\ol{\pi}_i^2 = -\ol{\pi}_i &1 \leq i \leq n-1\\
\ol{\pi}_i \ol{\pi}_{i+1} \ol{\pi}_i = \ol{\pi}_{i+1} \ol{\pi}_i \ol{\pi}_{i+1} &1 \leq i \leq n-2\\
\ol{\pi}_i \ol{\pi}_j = \ol{\pi}_j \ol{\pi}_i &|i-j| > 1.
\end{cases}
\]
Here $\ol{\pi}_i = \pi_i - 1$ for all $i$. 

In \cite{Nor} the author characterizes all simple and projective $H_n(0)$ modules. In particular, they are indexed by compositions $\alpha \vDash n$. The complete list of irreducible modules is given by $\{\bC_\alpha\}_\alpha$ where $\bC_\alpha$ is a one dimensional $H_n(0)$-module spanned by $v_\alpha$ where for $i = 1,\dots,n-1$,
\[
\pi_i v = 
\begin{cases}
0 & i \in D(\alpha)\\
1 & i \not\in D(\alpha)
\end{cases}.
\]
We will denote the complete list of projective modules by $\{\bP_\alpha\}$. For more details we refer the reader to \cite{Nor}. As a quick example we note that $\bC_n = \bP_n$ is the trivial representation where all generators $\pi_i$ act by $1$ because $D(n) = \{\}$.

\subsection{Function Theory}
Let $X = (x_1,x_2,\dots)$ be a totally ordered infinite set of variables. Then we denote the $\bZ$-algebra of {\em symmetric functions} in $X$ with coefficients in $\bZ$ by ${\rm Sym}$. There is a clear $\bN$-grading on this algebra, its degree $n$ component has basis given by the {\em Schur functions} $\{s_\lambda \; | \; \lambda \vdash n\}$, i.e. $\lambda$ is a partition of $n$. The Schur function $s_\lambda$ can be expressed as
\begin{equation} \label{schurDef}
s_\lambda = \sum_{T} \bx^T
\end{equation}
where the sum is over all semi-standard tableau $T$ of shape $\lambda$ and $\bx^T$ is the monomial
\[
\bx^T \coloneqq x_1^{c_1(T)} x_2^{c_2(T)} \cdots
\]
where $c_i(T)$ is the number of times $i$ appears in $T$. Given partitions $\mu \subset \lambda$, we also let $s_{\lambda/\mu} \in {\rm Sym}$ denote the {\em skew Schur function}. This function is defined via Equation \eqref{schurDef}, where we sum over all skew-tableau of shape $\lambda/\mu$. The most important example of a skew Schur function for us is a {\em ribbon Schur function}, where $\lambda/\mu$ is a ribbon tableau. We will index these by $s_\alpha$ where $\alpha$ is a composition corresponding to the ribbon $\lambda/\mu$. For further reading on symmetric functions we refer the reader to \cite{Mac}.

There is a coproduct structure on ${\rm Sym}$ given by replacing the variables $x_1,x_2,\dots$ with $x_1,x_2,\dots,y_1,y_2,\dots$ so that ${\rm Sym}$ becomes a graded Hopf algebra which is self dual under the basis $\{s_\lambda\}$ \cite[\S 2]{GrRe}.

There is a larger algebra which contains ${\rm Sym}$ given by loosening the requirement that the functions be symmetric. Once again if we let $X$ be a totally ordered set of variables, we can define the $\bZ$-algebra of {\em quasisymmetric functions} {\bf QSym} as the power series of bounded degree in $X$ which are upward shift invariant in the sense that the coefficient of the monomial $x_1^{\alpha_1} \cdots x_k^{\alpha_k}$ is equal to the coefficient of $x_{i_1}^{\alpha_1} \cdots x_{i_k}^{\alpha_k}$ for any increasing sequence of integers $i_1 < i_2 < \cdots < i_k$. Notice that every symmetric function satisfies this property, but there are quasisymmetric functions that are not symmetric, for example
\[
x_1 x_2^2 + x_1 x_3^2 + x_2 x_3^2
\]
is a quasisymmetric function in $3$ variables but is not symmetric. This means that we naturally have ${\bf Sym} \subset {\bf QSym}$. The algebra ${\bf QSym}$ has a basis consisting of {\em monomial symmetric functions} $M_\alpha$. For a composition $\alpha \vDash n$ with $\alpha = \{\alpha_1,\dots,\alpha_k\}$
\[
M_\alpha \coloneqq \sum_{i_1 < i_2 < \cdots < i_k} x_{i_1}^{\alpha_1} x_{i_2}^{\alpha_2} \cdots x_{i_k}^{\alpha_k}.
\]
We can use this basis to define one of the most important bases for ${\bf QSym}$, the {\em Gessel fundamental quasisymmetric functions}, denoted $\{F_\alpha\}$. To begin we let $F_0 = 1$ and for any composition $\alpha$,
\[
F_\alpha = \sum_{\beta \preceq \alpha} M_\beta
\]
where $\beta \preceq \alpha$ if we can obtain $\alpha$ by adding together adjacent elements of $\beta$. For example $(1,1,2,3,2) \preceq (2,2,5)$. This is sometimes called the {\em refinement order} on compositions.

The final algebra we consider is ${\bf NSym}$ the graded algebra of {\em noncommutative symmetric functions}. This is the free unital associative noncommutative algebra of noncommutative functions in $X$ invariant under the natural symmetric group action. Alternatively, we can define it as the free unital associative noncommutative algebra $\bZ\langle h_1, h_2, \dots \rangle$ generated over $\bZ$ by the symbols $h_1,h_2,\dots$ where $h_d$ has degree $d$. The degree $n$ component of ${\bf NSym}$ has $\bZ$-basis given by $\{\bh_\alpha \; | \; \alpha \vDash n\}$ where 
\[
\bh_\alpha = h_{\alpha_1} \cdots h_{\alpha_k}.
\]
Remember this is a noncommutative product, so considering compositions instead of partitions is important. One of the most important bases for the degree $n$ component of ${\bf NSym}$ consists of the noncommutative ribbon Schur functions, $\{s_\alpha \; | \; \alpha \vDash n\}$. We note that these are not the same as the ribbon Schur functions mentioned above, because they are not commutative, their commutative image will be one of the ribbon Schur functions. In terms of the $h_\alpha$ we define these as
\[
s_\alpha \coloneqq \sum_{\beta \preceq \alpha} (-1)^{\ell(\alpha) - \ell(\beta)} \bh_\beta.
\]
As in the symmetric function case there are coproducts on ${\bf QSym}$ and ${\bf NSym}$ but they are not self dual, instead they are dual to each other as Hopf algebras \cite[\S 5]{GrRe}.

\subsection{Quasicharacteristic Maps} \label{CharMaps}
Let $A$ be a finite dimensional algebra over a field $\bF$. We will often work with the {\em Grothendieck group} $\cG(A)$ of finitely generated $A$-modules. This group is the quotient of the free abelian group generated by isomorphism classes of finitely generated $A$-modules $[M]$ by the relation $[M] - [L] - [N]$ if there is a short exact sequence of finitely generated $A$-modules 
\[
0 \to L \to M \to N \to 0.
\]
If we let $A = H_n(0)$, $\cG(H_n(0))$ has free basis given by the collection of isomorphism classes of irreducible $H_n(0)$-modules, i.e. the $\{[C_\alpha]\}$ for all $\alpha \vDash n$. We will not use it much for now, but it is worth mentioning that there is another Grothendieck group $\cK(A)$ consisting of all finitely-generated projective $A$-modules which has free basis given by the collection of isomorphism classes of indecomposable projective $H_n(0)$-modules. These were also characterized in \cite{Nor} and correspond to compositions of $n$ as well.

In the following we assume some representation theory background, a good resource for these results is \cite{FH}. The symmetric group algebra $\bQ[\fS_n]$ is semisimple and has irreducible representations $V_\lambda$ indexed by partitions $\lambda \vdash n$. In this setting, the Grothendieck group $\cG(\bQ[\fS_\bullet])$ of the tower 
\[
\bQ[\fS_n] \colon \bQ[\fS_0] \injects \bQ[\fS_1] \injects \cdots
\]
of symmetric group algebras is the direct sum of the $\cG(\bQ[\fS_n])$ for all $n \geq 0$. It can be given the structure of a graded Hopf algebra where the product and coproduct are respectively induction and restriction of representations along the natural embedding $\fS_n \otimes \fS_m \injects \fS_{n+m}$. The {\em Frobenius characteristic map} ${\rm Fch}$ of a finite dimensional $\bQ[\fS_n]$-modules $V$ is defined first on the basis for $\cG(\bQ[\fS_n])$ consisting of isomorphism classes of Specht modules $[\bV_\alpha]$
\[
{\rm Fch}([\bV_\lambda]) = s_\lambda
\]
where $s_\lambda$ is the Schur function corresponding to the composition $\lambda$, we then extend linearly. Incredibly, the map ${\rm Fch}$ gives a graded Hopf algebra isomorphism between $\cG(\bQ[\fS_\bullet]) \cong {\rm Sym}$ \cite[\S 4.4]{GrRe}. This connection has been incredibly useful as it allows us to study the representation theory of the symmetric group by working with symmetric functions and vice versa. One important result coming from this connection is the decomposition of tensor products $V_\lambda \otimes V_{n}$ of Specht modules considered now as a representation of $\fS_{|\lambda|+n}$ into its irreducible components via the Pieri rule. For further details we refer the reader to \cite[I 5]{Mac}.

It turns out that there are two analogous characteristic maps ${\rm Ch}$ and ${\bf ch}$ defined by Krob and Thibon \cite{KT}, which make it possible to study representations of $H_n(0)$ through the rings ${\bf QSym}$ and ${\bf NSym}$ defined above. We will recall their construction because it will be very important in the coming sections.

As discussed, the two Grothendieck groups $\cG(H_n(0))$ and $\cK(H_n(0))$ have bases given by $\{[\bC_\alpha] \; | \; \alpha \vDash n\}$ and $\{[\bP_\alpha] \; | \; \alpha \vDash n\}$ respectively. Similar to the symmetric group algebra case, we have Grothendieck groups $\cG(H_\bullet(0)) \coloneqq \bigoplus_{n \geq 0} \cG(H_n(0))$ and $\cK(H_\bullet(0)) \coloneqq \bigoplus_{n \geq 0} \cK(H_n(0))$ associated to the tower of algebras
\[
H_\bullet(0) \colon H_0(0) \injects H_1(0) \injects H_2(0) \injects \cdots.
\]
These groups are both graded Hopf algebras with product and coproduct given by induction and restriction of representations along the natural embedding $H_n(0) \otimes H_m(0) \injects H_{n+m}(0)$. Furthermore, they are dual to each other via the pairing $\langle [\bP_\alpha], [\bC_\beta] \rangle = \delta_{\alpha,\beta}$ where this is the Kronecker-Delta function.

Knob and Thibon \cite{KT} define two linear characteristic maps
\[
{\rm Ch} \colon \cG(H_\bullet(0)) \to {\bf QSym} \qquad \text{and} \qquad {\bf ch} \colon \cK(H_\bullet(0)) \to {\bf NSym},
\]
by ${\rm Ch}([\bC_\alpha]) \coloneqq F_\alpha$ and ${\bf ch}(\bP_\alpha) \coloneqq s_\alpha$ where $F_\alpha$ is the fundamental quasisymmetric function and $s_\alpha$ is the skew-Schur function. They then show these maps are graded isomorphisms of Hopf algebras. For more information we refer the reader to \cite{KT}.

\subsection{The $0$-Hecke Category $\cH$} \label{mapSec}
In \cite{Big} the author describes a natural way to view $H_n(0)$ as a quotient of the braid group $B_n$. Recall that a {\em geometric braid} is a disjoint union of $n$ edges called {\em strands}, in $D \times I$ where $I = [0,1]$ and $D$ is a closed disk. The set of endpoints of the strands is required to be $\{p_1,\dots,p_n\} \times \{0,1\}$, and each strand is required to intersect each disk cross-section exactly once. Two geometric braids are said to be equivalent if it is possible to continuously deform one to get the other. The elements of $B_n$ are equivalence classes of braids under this continuous deformation equivalence relation. For more we refer the reader to \cite{Big}. From now on, we will refer to geometric braids as braids.

\begin{definition}
A crossing in a braid is called {\em positive} if the strand on top in the crossing goes from top left to bottom right, otherwise it is called {\em negative}.
\end{definition}

We can encode each $\pi_i$ as a braid with strands connecting vertices $i$ to $i+1$ in the top row to vertices $i+1$ to $i$ in the bottom row with a positive crossing. Then every element of $H_n(0)$ can be represented as a composition of positive crossing diagrams and hence can be viewed as a braid diagram with strictly positive crossings, meaning the strand originating at the first vertex crosses above any vertex it crosses. The strand originating from the second vertex can only cross below the strand originating from vertex $1$ and above all other strands etc.

\begin{definition}
We call a braid diagram from $[n]$ vertices to $[m]$ vertices {\em order preserving} if there are no crossings.
\end{definition}

We will define $\cH$ as the quotient of the Braid Category in \cite[\S1.2]{WR} by the skein relation, i.e. that $\pi_i^2 = \pi_i$. Graphically, this is the relation
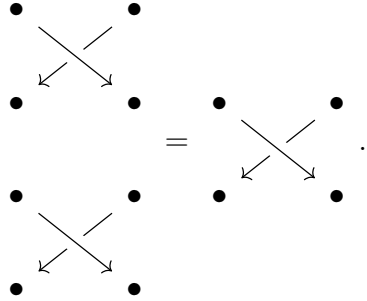
\[
\begin{tikzcd}
\bullet & \bullet \ar[dl]\\
\bullet & \bullet \ar[from = ul, crossing over]\\
\bullet & \bullet \ar[dl] \\
\bullet & \bullet\ar[from = ul, crossing over]
\end{tikzcd}
= 
\begin{tikzcd}
\bullet & \bullet \ar[dl] \\
\bullet & \bullet\ar[from = ul, crossing over]
\end{tikzcd}.
\]
More explicitly, the objects of $\cH$ consist of sets $[n]$ for $n \in \bZ_{\geq 0}$ and $\hom_{\cH}([n],[m])$ for $n \leq m$ is given by all positive crossing braid diagrams from $n$ vertices to $m$ vertices.

To compose two maps $f \in \hom_{\cH}([n],[m])$ and $g \in \hom_\cH([m],[\ell])$ we view both maps as elements of $H_\ell(0)$ by padding the top rows of the corresponding diagrams with {\em dummy vertices}, called ``free" ends in \cite{WR}, so that every row has $\ell$ vertices. We then connect the dummy vertices to the vertices in the bottom row that have degree $0$ so that we do not introduce any crossings between the edges we introduce and so that every crossing we must create is positive. The edges adjacent to dummy vertices are called {\em dummy edges}. Now compose as we would in $H_\ell(0)$, i.e. by following paths and applying the skein relations to ensure we have positive crossings. We then remove all dummy vertices and the dummy edges from the resulting composition.

Notice that the number of dummy vertices we add to $f$ does not matter after we add $m-n$ because after that each dummy vertex is mapped identically to itself.

\begin{example}
As an example consider the composition
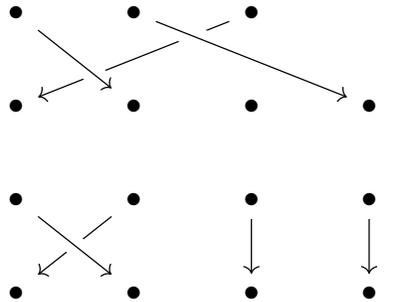
\[
\begin{tikzcd}
\bullet & \bullet & \bullet \ar[dll] &\\
\bullet & \bullet \ar[from = ul, crossing over] & \bullet & \bullet \ar[from = ull, crossing over]\\
\bullet & \bullet \ar[dl] & \bullet \ar[d] & \bullet \ar[d]\\
\bullet & \bullet \ar[from = ul, crossing over] & \bullet & \bullet
\end{tikzcd}.
\]
We only have to place a dummy vertex in the first row of the first diagram. Dummy vertices will be denoted by $\odot$ and dummy edges will be denoted by dashed edges. So we have
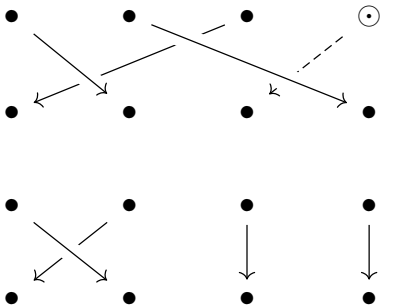
\[
\begin{tikzcd}
\bullet & \bullet & \bullet \ar[dll] & \odot \ar[dl, dashed] \\
\bullet & \bullet \ar[from = ul, crossing over] & \bullet & \bullet \ar[from = ull, crossing over]\\
\bullet & \bullet \ar[dl] & \bullet \ar[d] & \bullet \ar[d]\\
\bullet & \bullet \ar[from = ul, crossing over] & \bullet & \bullet
\end{tikzcd}.
\]
We can then decompose this in terms of the generators $\pi_i$ for $H_4(0)$. The top map is exactly $\pi_3 \pi_1 \pi_2$, the final map is $\pi_1$, so the composition is $\pi_1 \pi_3 \pi_1 \pi_2 = \pi_3 \pi_1^2 \pi_2 = \pi_3 \pi_1 \pi_2$. When we remove the dummy vertex and the connected edge we get back the original diagram.
\end{example}

The following are an important list of relations that result from the above composition rules:
%
%
\begin{equation} \label{rel1}
\begin{tikzcd}
\bullet \ar[d] & \bullet \ar[dr] &\\
\bullet & \bullet & \bullet\\
\bullet & \bullet \ar[dl] &\bullet \ar[d]\\
\bullet & \bullet \ar[from = ul, crossing over] &\bullet
\end{tikzcd} \qquad
= \qquad
\begin{tikzcd}
\bullet \ar[dr] & \bullet \ar[dr] &\\
\bullet & \bullet & \bullet\\
\end{tikzcd}
\end{equation}
\ \\
\begin{equation} \label{rel2}
\begin{tikzcd}
\bullet \ar[dr] & \bullet \ar[dr] &\\
\bullet & \bullet & \bullet\\
\bullet & \bullet \ar[dl] &\bullet \ar[d]\\
\bullet & \bullet \ar[from = ul, crossing over] &\bullet
\end{tikzcd} \qquad
= \qquad
\begin{tikzcd}
\bullet \ar[dr] & \bullet \ar[dr] &\\
\bullet & \bullet & \bullet\\
\end{tikzcd}
\end{equation}
\ \\
\begin{equation} \label{rel3}
\begin{tikzcd}
\bullet \ar[drr] & \bullet \ar[drr] & &\\
\bullet & \bullet & \bullet & \bullet\\
\bullet & \bullet \ar[dl] &\bullet \ar[d] & \bullet \ar[d]\\
\bullet & \bullet \ar[from = ul, crossing over] &\bullet & \bullet
\end{tikzcd} \qquad
= \qquad 
\begin{tikzcd}
\bullet \ar[drr] & \bullet \ar[drr] & &\\
\bullet & \bullet & \bullet & \bullet
\end{tikzcd}
\end{equation}
\begin{equation} \label{rel4}
\begin{tikzcd}
\bullet \ar[d] & \bullet \ar[d] &\\
\bullet & \bullet & \bullet\\
\bullet & \bullet \ar[dl] &\bullet \ar[d]\\
\bullet & \bullet \ar[from = ul, crossing over] &\bullet
\end{tikzcd} \qquad
= \qquad
\begin{tikzcd}
\bullet & \bullet \ar[dl] &\\
\bullet & \bullet \ar[from = ul, crossing over] &\bullet
\end{tikzcd}\qquad
= \qquad
\begin{tikzcd}
\bullet & \bullet \ar[dl]\\
\bullet & \bullet \ar[from = ul, crossing over] \\
\bullet \ar[d] & \bullet \ar[d] &\\
\bullet & \bullet & \bullet\\
\end{tikzcd}.
\end{equation}

\begin{remark}
It is easy to check that all of these relations arise from our padded composition in $\cH$. We will illustrate how to check Relation \eqref{rel2} and leave the rest for the reader. Once again, dummy vertices are denoted by $\odot$ and dummy edges are dashed, Relation \eqref{rel2} can be seen as,
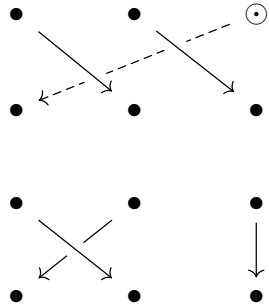
\[
\begin{tikzcd}
\bullet & \bullet& \odot \ar[dll, dashed]\\
\bullet & \bullet \ar[from = ul, crossing over] & \bullet \ar[from = ul, crossing over] \\
\bullet & \bullet \ar[dl] &\bullet \ar[d]\\
\bullet & \bullet \ar[from = ul, crossing over] &\bullet
\end{tikzcd}.
\]
In $H_3(0)$ this is $\pi_1 \pi_1 \pi_2 = \pi_1 \pi_2$ by the Skein relation. Recall that composition is read from right to left and with our diagrams we read from top to bottom. As a result when we remove the dummy vertex and its connected edge we recover the original diagram.
\end{remark}

%
%

\begin{proposition} \label{OIsub}
${\bf OI}$ is a subcategory of $\cH$.
\end{proposition}

\begin{proof}
We have a natural embedding of objects. For maps, every order preserving injection corresponds to an order preserving braid diagram. It remains to check that given two order preserving maps, they compose as they would in ${\bf OI}$ using our composition rules. 

Let $f$ and $g$ be two order preserving maps, $f \in \hom_{\cH}([n],[m])$, $g \in \hom_{\cH}([m],[\ell])$. To compose we will pad $f$ with $\ell-n$ dummy vertices and $g$ with $\ell-m$ dummy vertices. When we add the dummy edges by virtue of the fact that the original maps were order preserving there will never be crossings between non-dummy edges in the resulting composition in $H_\ell(0)$. So when we remove the dummy vertices and edges, we will recover the exact composition we would expect in ${\bf OI}$.
\end{proof}

\begin{example}
Consider the following composition of order preserving maps
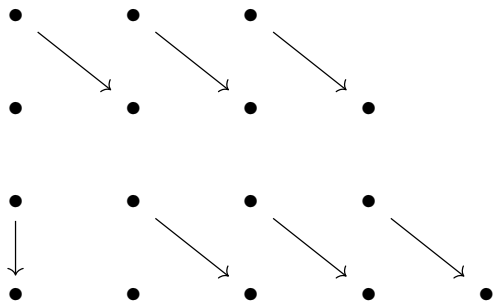
\[
\begin{tikzcd}
\bullet &\bullet & \bullet &&\\
\bullet & \bullet \ar[from=ul, crossing over] & \bullet \ar[from=ul, crossing over] & \bullet \ar[from=ul, crossing over] &\\
\bullet \ar[d]& \bullet & \bullet & \bullet & \\
\bullet & \bullet & \bullet \ar[from = ul, crossing over] & \bullet \ar[from = ul, crossing over]& \bullet \ar[from = ul, crossing over]
\end{tikzcd}
\]
After padding with dummy edges and vertices we have,
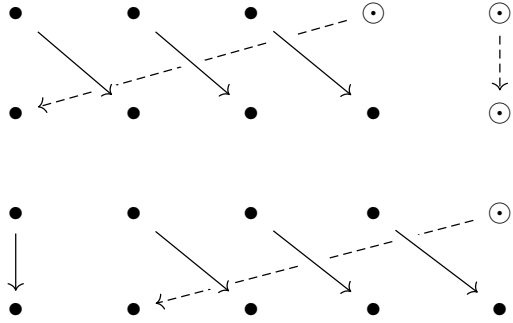
\[
\begin{tikzcd}
\bullet &\bullet & \bullet & \odot \ar[dlll, dashed]& \odot \ar[d, dashed] \\
\bullet & \bullet \ar[from=ul, crossing over] & \bullet \ar[from=ul, crossing over] & \bullet \ar[from=ul, crossing over] & \odot \\
\bullet \ar[d]& \bullet & \bullet & \bullet & \odot \ar[dlll, dashed] \\
\bullet & \bullet & \bullet \ar[from = ul, crossing over] & \bullet \ar[from = ul, crossing over]& \bullet \ar[from = ul, crossing over]
\end{tikzcd}
\]
When we follow all of the edges we end up with the following element of $H_5(0)$,
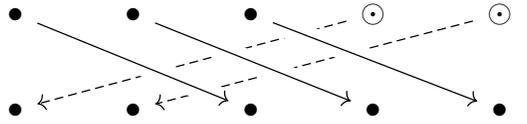
\[
\begin{tikzcd}
\bullet &\bullet & \bullet & \odot \ar[dlll, dashed]& \odot \ar[dlll, dashed] \\
\bullet & \bullet & \bullet \ar[from = ull, crossing over] & \bullet \ar[from = ull, crossing over]& \bullet \ar[from = ull, crossing over]
\end{tikzcd}.
\]
After removing the dummy vertices and edges we recover the expected order preserving map.
\end{example}

\begin{proposition}
Composition in $\cH$ is well defined and associative.
\end{proposition}

\begin{proof}
One way to see this is that composition is well defined in the Braid Category and hence is well defined in the quotient. This also follows immediately from the fact that composition is well defined and associative in $H_n(0)$ for any $n$. Since when we compose two maps $f \in \hom_{\cH}([n],[m])$ and $g \in \hom_{\cH}([m],[\ell])$ we first pad them and proceed to compose the corresponding padded diagrams in $H_\ell(0)$. Associativity in $H_\ell(0)$ implies that no matter how we compose we get the same diagram, so when we remove dummy vertices and edges we get the same well defined diagram in $\hom_{\cH}([n],[\ell])$.
%
\end{proof}

\begin{definition}
We call the natural injection of $H_n(0)$ into $H_m(0)$ sending $\pi_i \mapsto \pi_i$ for $i = 1,\dots, n-1$ the {\em principal injection} and denote it by $\iota_{n,m}$. This is represented by the order preserving braid diagram that sends vertex $i$ to vertex $i$.
\end{definition}

Notice that $H_m(0)$ acts on $\hom_{\cH}([n],[m])$ via post-composition. We can see that the element $\iota_{n,m}$ generates all other elements under this action, but the action is not transitive as in the ${\bf FI}$ case, which we would expect because there are no inverses in $H_m(0)$. We can also see that the annihilator of $\iota_{n,m}$ is isomorphic to $H_{m-n}(0)$ because it is generated by all the maps $\ol{\pi}_{n+1}, \dots, \ol{\pi}_{m}$. From this we can conclude,

\begin{proposition}
As a $H_m(0)$-module,
\[
\hom_{\cH}([n],[m]) \cong H_m(0)/H_{m-n}(0).
\]
\end{proposition}

\begin{proof}
Define a $H_m(0)$-equivariant map
\[
\phi \colon H_m(0) \to \hom_{\cH}([n],[m])
\]
via the action of the element on $\iota_{n,m}$, so $\phi(\pi_i) = \pi_i \iota_{n,m} $ and we extend so $\phi$ is an algebra homomorphism. From the above observations, this map is surjective. By the first isomorphism theorem we have
\[
H_m(0) / \ker(\phi) \cong \hom_{\cH}([n],[m]).
\]
But by definition $\ker(\phi)$ is exactly the annihilator of $\iota_{n,m}$, and this is isomorphic to $H_{m-n}(0)$.
\end{proof}

Using this, we can derive a $\cH$-module criterion,

\begin{theorem}[$\cH$-module criterion] \label{criterion}
Suppose that $\{W_n\}$ is a sequence of $H_n(0)$ representations with $H_n(0)$-equivariant maps $\phi_n \colon W_n \to W_{n+1}$, where the action on $W_{n+1}$ is given by embedding $H_n(0)$ into $H_{n+1}(0)$. Let $K \cong H_{m-n}(0)$ be the stabilizer of $\iota_{n,m}$ under the action of $H_m(0)$ by post composition. Then $\{W_n\}$ can be promoted to a $\cH$-module with $(\iota_{n,n+1})_\ast = \phi_n$ if and only if for all $n < m$
\[
\pi  \cdot v = v \qquad \qquad \forall \pi \in K \qquad \text{and} \qquad v \in \im((\iota_{n,m})_\ast).
\]
\end{theorem}

\begin{proof}
First if we assume that $\{W_n\}$ is a $\cH$-module, we clearly have $\pi \cdot v = v$ for every $\pi \in K$ and $v \in \im((\iota_{n,m})_\ast)$ by the rules for composition, since we know the stabilizer will be generated by $\pi_{n+1},\dots, \pi_{m}$ and all of these element act by the identity after $\iota_{n,m}$ so by functoriality we must have $\pi_\ast (\iota_{n,m})_\ast= (\pi \cdot \iota_{n,m})_\ast  = (\iota_{n,m})_\ast$.

For the other direction, we can factor any map $f \in \hom_{\cH}([n],[m])$ as $f = T_{w} \iota_{n,m}$ where $T_w$ is a word in the generators of $H_{m}(0)$. We will define the $\cH$ action by letting $(\iota_{n,n+1})_\ast = \phi_n$ and then $f_\ast = T_{w} \phi_{m-1} \cdots \phi_n$. It remains to check that functoriality is satisfied.

Suppose we have $f \in \hom_{\cH}([n],[m])$ and $g \in \hom_{\cH}([m],[\ell])$, then we claim 
\[
(g \circ f)_\ast =T_{p} \phi_{\ell-1} \cdots \phi_m T_{w} \phi_{m-1} \cdots \phi_n = g_\ast f_\ast.
\]
In order for the above to be well defined and true we need any permutation of strictly dummy vertices in the composition $g \circ f$ to act by the identity on any vector in the image of $(\iota_{n,\ell})_\ast$. This is guaranteed by our assumption because such permutations are exactly the stabilizer of $\iota_{n,\ell}$.
\end{proof}

\begin{remark}
We quickly notice that checking the above is equivalent to checking that every element that annihilates $\iota_{n,m}$ acts by $0$. This is because if $\pi$ stabilizes, then $\pi-1$ annihilates and vice versa.
\end{remark}

\noindent This gives us an easy way to check if a sequence of $H_n(0)$-modules can be realized as a $\cH$-module. We will now prove a few more properties of this category. First, we will show it is noetherian.

\begin{definition}
We say that a category of modules is {\bf locally noetherian} if any submodule of a finitely generated module is finitely generated.
\end{definition}

\begin{definition}
Given two categories $\cC$ and $\cC'$, we say that a functor $\Phi \colon \cC \to \cC'$ satisfies {\bf property (F)} if the following condition holds: given any object $x$ of $\cC'$ there exist finitely many objects $y_1,\dots,y_n$ of $\cC$ and morphisms $f_i \colon x \to \Phi(y_i)$ in $\cC'$ such that for any object $y$ of $\cC$ and any morphism $f \colon x \to \Phi(y)$ in $\cC'$, there exists a morphism $g \colon y_i \to y$ in $\cC$ such that $f = \Phi(g) \circ f_i$.
\end{definition}

This property is important because it is equivalent to the pullback functor $\Phi^\ast \colon {\rm Rep}_\bk(\cC') \to {\rm Rep}_{\bk}(\cC)$ preserving finite generation. For more on this we direct the reader to \cite[\S 3]{SS1}.

\begin{theorem} \label{HNoeth}
The category of $\cH$-modules, ${\rm Mod}_{\cH}$, is noetherian.
\end{theorem}

\begin{proof}
We use Gr\"obner methods outlined in \cite{SS1}. In particular, we will show that $\cH$ is quasi-Gr\"obner which proves that it is noetherian. In \cite{SS1} the authors show that ${\bf OI}$ is a Gr\"obner category, i.e. that it is noetherian. So it suffices to produce a functor $\Phi \colon {\bf OI} \to \cH$ that satisfies property (F). The functor $\Phi$ will be the natural inclusion outlined in Proposition \ref{OIsub}.

It is clear that $\Phi$ is essentially surjective, so it remains to show that it satisfies property (F). Let $x = [n]$ be a given object in $\cH$. Every map $f \in \hom_{\cH}([n],[m])$ can be factored as $i(n,m) T_w$ where $T_w \in H_n(0)$ and $i(n,m)$ is an order preserving injection from $[n]$ to $[m]$. If we take $y_1,\dots,y_{n!} = [n]$ and $f_i \colon [n] \to \Phi(y_i)$ to be the $i$th element of $H_n(0)$ under any enumeration it is clear that $\Phi$ satisfies property (F).
\end{proof}

\section{Induced Modules} \label{InducedSec}
We will now consider the {\em induced modules} $M(\bC_\alpha)$ where $\bC_\alpha$ is the simple one dimensional $H_n(0)$-module defined in the Section \ref{NotationSec}. $M(\bC_\alpha)$ is defined as
\[
M(\bC_\alpha)_m \coloneqq \bk[\hom_{\cH}([n],[m])] \otimes_{\bk[H_n(0)]} \bC_\alpha,
\]
where the action of $\cH$ is composition on the left tensor factor. It is not hard to see that
\[
M(\bC_\alpha)_m \cong {\rm Ind}_{H_n(0) \times H_{m-n}(0)}^{H_m(0)} \bC_\alpha \boxtimes \bC_{m-n}.
\]
We also define $M(n)_m \coloneqq  \bk[\hom_{\cH}([n],[m])]$. There is a convenient way to index basis elements of $M(\alpha)$ in each degree.

\begin{lemma} \label{basis}
For a fixed $\alpha \vDash n$ and degree $d$ a basis for $M(\alpha)_d$ is indexed by all order preserving injections of $n$ into $d$. Explicitly this basis is given by $\{g \otimes v_\alpha\}$ where $g$ ranges over all order preserving injections of $n$ into $d$ and $v_\alpha$ spans $\bC_\alpha$.
\end{lemma}

\begin{proof}
For fixed $d$ and $n$, we may assume $d \geq n$, the space $M(\alpha)_d$ is exactly
\[
M(n)_d \otimes_{\bk[H_n(0)]} \bC_\alpha.
\]
Pick any element $f \otimes v_\alpha$ where $v_\alpha$ spans $\bC_\alpha$. We may assume by linearity that $f$ is a single map and not a sum of maps. If $f$ has any crossings, we can factor $f$ as an element of $H_n(0)$ followed by an order preserving injection $g$, i.e. $f = g T_w$. We can then pass the element of $H_n(0)$ through the tensor and have it act on $v_\alpha$. This implies that this vector $f \otimes v_\alpha = g \otimes (T_w v_\alpha)$. But we know $T_w$ acts by either $0$ or $1$ on $v_\alpha$ so either the original vector is zero or it equals $g \otimes v_\alpha$ for some order preserving injection $g$.
\end{proof}

\begin{corollary} \label{dim}
For a fixed $\alpha \vDash n$, $\dim\left(M(\alpha)_d\right) = \binom{d}{n}$.
\end{corollary}

\begin{proof}
This follows immediately from Lemma \ref{basis} because there are $\binom{d}{n}$ order preserving injections of $[n]$ into $[d]$.
\end{proof}


Under the quasisymmetric characteristic map ${\rm Ch} \colon \cG(\cH) \to {\rm QSym}$ defined in Section \ref{NotationSec}, the isomorphism class of the simple element $[\cC_\alpha]$ corresponds to the fundamental quasisymmetric function $F_\alpha$. We can use this characteristic map to decompose the isomorphism class $[M(\bC_\alpha)]_n$, in particular it corresponds to the analogous variant of the Pieri rule for fundamental quasisymmetric functions which we now recall.


Given two words $w$ and $v$ of length $n$ and $m$ respectively, we define their {\em shuffle set}, denoted by $w \odot v$ as the set of all words, $u_1\cdots u_{m+n}$, where for some subset of disjoing indices $\{i_1,\dots,i_n\}$ and $\{j_1,\dots,j_m\}$ with $i_1 < \cdots < i_n$ and $j_1 < \cdots < j_m$, we have $u_{i_k} = w_k$ and $u_{j_k} = v_k$. Given a word $w$ on an ordered set such as $\bZ_{\geq 0}$ we define the {\em descent set} of $w$ as
\[
D(w) = \{i \; | \; w_i > w_{i+1} \}.
\]
As a quick example consider the word $w = 14253$, then $D(w) = \{2,4\}$.

For fundamental quasisymmetric functions $F_\alpha$ and $F_\beta$ we have the following multiplication rule, for more details on this rule we refer the reader to \cite{LP},
\[
F_\alpha F_\beta = \sum_{u \in w(\alpha) \odot w(\beta)} F_{C(D(u))}
\]
where $w(\alpha)$ is some word on $\bZ_{\geq 0}$ with descent set $D(w(\alpha)) = D(\alpha)$ and the same for $w(\beta)$ with the condition that $w(\alpha)$ and $w(\beta)$ must consist of disjoint natural numbers. We recall that $C(D(u))$ is the composition corresponding to the descent set of the word $u$.

We are specifically concerned with the case where $\beta = m$. In this case, combinatorially if we interpret $\alpha \vDash n$ as a ribbon diagram we sum over all $F_\gamma$ with $\gamma \vDash m+n$ a ribbon diagram constructed from $\alpha$ by adding $m$ boxes in any way to any of the positions of $\alpha$. This multiplication rule allows us to discover which simple modules are in each degree of our induced modules. Consider the following example,

\begin{example}
Consider the following induced $\cH$-module, $M({\tiny \young(:2,13)})$. We will explicitly compute the simples in each degree. The first nonzero degree is $3$,
\[
M({\tiny \young(:2,13)})_3 = \young(:2,13)\;.
\]
We number the boxes to illustrate the connection to descent words and the multiplication rule above. Each diagram corresponds to the descent word where we read from bottom to top, left to right. So degree $3$ correspond to when $\beta = \emptyset$ and so we only have the descent word $132$ corresponding to the composition $(2,1)$. Next in degree $4$ we have
\[
M({\tiny \young(:2,13)})_4 = \young(:1,:3,24) + \young(:13,24:)+  \young(:3,14,2:) + \young(::3,124:) \;.
\]
Here $w(\beta) = 1$ and $w(\alpha) = 243$. We then sum over all shuffles, so we set the descent words $\{2431,2413,2143,1243\}$ which correspond to inserting $1$ into all possible places. Now in degree $5$ we will have $\binom{5}{3} = 10$ diagrams,
\begin{align*}
M({\tiny \young(:2,13)})_5 = \; & \young(:12,:4:,35:) + \young(::2,:14,35:) +\young(:124,35::) + \young(:2,:4,15,3:) + \young(:24,15:,3::)  \\ &+ \young(::4,125,3::)+ \young(::2,::4,135)+ \young(::24,135:) + \young(::4,:25,13:) + \young(:::4,1235:) \; .
\end{align*}
We want to point out here that the same shapes correspond to the same simple module, so just because they have different labelings does not make them different modules. \qedhere
\end{example}

From this point forward, we will no longer number the boxes because the descent set is recoverable from the diagrams themselves. We placed the numbers in the above example purely to illustrate the connection between the multiplication rule for fundamental quasisymmetric functions and adding boxes to certain positions of $\alpha$.

Furthermore, we note that the action of the $0$-Hecke algebra $H_n(0)$ on the simples in the $\cH$-degree $n$ component of a $\cH$-module $M(\alpha)$ correspond to shifting added boxes down to a lower position in the diagram. This perspective will be very useful in the coming sections. We will often switch perspectives depending on which allows for the easiest proof. To see what we mean, consider the following example.

\begin{example}
Once again if we consider the induced module $M((2,1))$, we can see that each of the diagrams in the above example correspond to a basis element, which are indexed by order preserving injections. In degree $4$ we have
\[
M({\tiny \young(:~,~~)})_4 = \young(:\bullet,:~,~~) + \young(:\bullet~,~~:) +  \young(:~,\bullet~,~:) + \young(::~,\bullet~~:) \;,
\]
where the marked boxes are the box we add to the ribbon tableau corresponding to $(2,1)$. These diagrams correspond respectively to the basis elements
{\tiny \[
\begin{tikzcd}
\bullet \ar[d]& \bullet \ar[d]& \bullet \ar[d]&\\
\bullet & \bullet & \bullet & \bullet
\end{tikzcd}
\otimes v_{(2,1)}, \quad
\begin{tikzcd}
\bullet \ar[d]& \bullet \ar[d]& \bullet \ar[dr]&\\
\bullet & \bullet & \bullet & \bullet
\end{tikzcd}
\otimes v_{(2,1)},
\]
\[
\begin{tikzcd}
\bullet \ar[d]& \bullet \ar[dr]& \bullet \ar[dr]&\\
\bullet & \bullet & \bullet & \bullet
\end{tikzcd}
\otimes v_{(2,1)}, \quad
\begin{tikzcd}
\bullet \ar[dr]& \bullet \ar[dr]& \bullet \ar[dr]&\\
\bullet & \bullet & \bullet & \bullet
\end{tikzcd}
\otimes v_{(2,1)}
\]}
\noindent where we read from left to right then top to bottom and $v_{(2,1)}$ is the basis element that spans $\bC_{(2,1)}$. To see this correspondence, notice that the final diagram is preserved by the $H_4(0)$ action where $\pi_1$ acts by $1$, $\pi_2$ acts by $1$, and $\pi_3$ acting by $0$. The corresponding descent set is precisely $\{3\}$ and so the corresponding composition of $4$ is $(3,1)$. Notice this is the diagram where the box is placed in the lowest possible position. 

When we quotient by this simple, one can check that the next simple that we can inject into this module is precisely the third diagram where we add a box to the second position. Under the $H_4(0)$ action, we can map the third diagram to the final diagram where we add a box to the $1$st position by applying $\pi_1$. In this same way, it is possible to move any box we have added in a higher row to a lower row under this $H_4(0)$ action since adding boxes to the $i$th row corresponds to having that many degree $0$ vertices before the $i$th strand. This implies that if a submodule contains a simple in $M(\alpha)$ that corresponds to adding a box to a higher position of $\alpha$, the submodule must also contain all the simple modules corresponding to adding the same boxes to lower positions. This will be extremely useful in the coming sections.
\end{example}

\begin{proposition} \label{HSimples}
The simple objects in ${\rm Mod}_{\cH}$ are precisely the modules equal to $C_\alpha$ in degree $|\alpha|$ and zero elsewhere.
\end{proposition}

\begin{proof}
Suppose $V$ is a simple $\cH$-module. Let $d$ be the first degree where $V_d \not= 0$. Then by simplicity we must have $V_{\geq d+1} = 0$, otherwise $V_{\geq d+1}$ would be a nonzero submodule. Furthermore, $V_d$ must be a simple $H_{d}(0)$ module otherwise it would have a submodule. This means $V_d = \bC_\alpha$ for some composition $\alpha \vDash d$.
\end{proof}

%

%
%
%

\section{Padded Induced Modules} \label{PaddedInducedSec}
We say that a finitely generated $\cH$-module $V$ has {\em polynomial degree} $d$ if $\dim(V_n)$ is given by a degree $d$ polynomial for sufficiently large $n$.

Given a composition $\alpha \vDash n$, for $0 \leq i \leq n$, let $\eta_i(\alpha)$ consist of the first $i$ blocks of $\alpha$ and $\tau_j(\alpha)$ consist of the last $j$ blocks of $\alpha$. We refer to these $\eta_i(\alpha)$ as {\em heads} of $\alpha$ and $\tau_j(\alpha)$ as {\em tails} of $\alpha$. We explain what a head looks like explicitly and the tail is similar. If $\alpha = (\alpha_1,\dots,\alpha_k)$ then $\eta_i(\alpha) = (\alpha_1,\dots, \alpha_{\ell} - \beta)$ where $\beta$ is a potential correction term to make sure that $\eta_i(\alpha)$ is a composition of $i$. We say that $\beta$ and $\gamma$ are {\em summands} of $\alpha$ if we can write $\alpha = \gamma + \beta$ where $\gamma$ and $\beta$ are both compositions. We say that $\eta_i(\alpha)$ or $\tau_j(\alpha)$ is {\em split} if we need a correction term $\beta$, combinatorially this is if the $i$th block reading from left to right is not at the end of a row. If $\eta_i(\alpha)$ is split then $\alpha = \eta_i(\alpha) + \tau_{n-i}(\alpha)$, if not then $\alpha = \eta_i(\alpha) \lhd \tau_{n-i}(\alpha)$.

We now define a new module $M(\alpha,k)$ where $\alpha$ is a composition of $n$ and $k$ is an integer $-1 \leq k \leq n$. First we let $M(\alpha,n) \coloneqq M(\alpha)$ and $M(\alpha,-1) \coloneqq \bC_\alpha$ these are both quotients of $M(\alpha)$. Then for $0 \leq k < n$, we let $M(\alpha,k)$ will be the quotient of $M(\alpha)$ of degree $k$ defined as follows. Let $A_{\alpha,k}$ be the submodule of $M(\alpha)$ generated by the basis element in degree $|\alpha|+1$ indexed by the order preserving injection $g_k \colon [n] \to [n+1]$,
\[
g_k(i) =
\begin{cases}
i+1 &  i \geq |\alpha|-k\\
i &  \text{else}
\end{cases}.
\]
When we quotient by this module, the remaining simples in degree $m$ of $M(\alpha,k)$ are indexed by $M(\tau_k(\alpha))_{m-n+k}$. To find the simples of $M(\alpha,k)$ corresponding to the compositions $\{\beta_{i,m-n+k}\}_i$ of $M(\tau_k(\alpha))_{m-n+k}$ we either get $\eta_{n-k}(\alpha) + \beta_{i,m-n+k}$ or $\eta_{n-k}(\alpha) \lhd \beta_{i,m-n+k}$. The first case occurs if $\eta_{n-k}(\alpha)$ is split or if we add a box to the first row of $\tau_k(\alpha)$ in $\beta_{i,m-n+k}$. The second case occurs if $\eta_{n-k}(\alpha)$ is not split and we do not add a box to the first row of $\tau_k(\alpha)$. We explain this further below in Lemma \ref{polySimples}. We call these {\em degree k induced $\alpha$ modules}. 

\begin{proposition}
The module $M(\alpha,k)$ where $|\alpha| = n$, has Hilbert function $\binom{d-n+k}{k}$ and so has polynomial degree $k$.
\end{proposition}

\begin{proof}
The simples that occur in $M(\alpha,k)_d$ are precisely those obtained from $\alpha$ by adding $d-n$ boxes to the top $k+1$ positions of $\alpha$. There are $\binom{d-n+k}{k}$ ways to do this and each simple is one dimensional.
\end{proof}

\begin{remark}
Notice, this result recovers the previous statement that $\dim(M(\alpha)_d) = \binom{d}{n}$, as $M(\alpha) = M(\alpha,n)$ if $|\alpha| = n$. Comparing the proof of this result with Corollary \ref{dim} further illustrates the connection between the basis of order preserving injections and the combinatorial interpretation of computing simples that occur using ribbon diagrams.
\end{remark}

\begin{example} \label{quotient1}
Consider the example of $M((2,2),3)$. We define this as a quotient of $M((2,2))$ by the submodule generated by
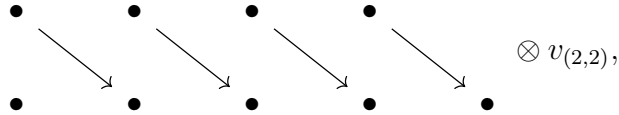
\[
\begin{tikzcd}
\bullet \ar[dr] & \bullet \ar[dr]& \bullet \ar[dr]& \bullet \ar[dr]&\\
\bullet & \bullet &\bullet &\bullet &\bullet
\end{tikzcd}
\otimes v_{(2,2)},
\]
where recall $v_{(2,2)}$ spans $\bC_{(2,2)}$. The above map is precisely the map indexed by the order preserving injection $g$ described above. In this case $|\alpha|-3 = 1$. Notice that $\alpha_1$ is not split, i.e. $\alpha_1 = (1)$ and $\alpha^3 = (1,2)$. From the above all the remaining simples should be generated by adjoining the simples in $M((1,2))$ to $\alpha_1$. In degree $4$ this quotient is spanned by
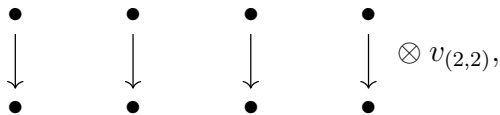
\[
\begin{tikzcd}
\bullet \ar[d] & \bullet \ar[d]& \bullet \ar[d]&\bullet \ar[d]\\
\bullet & \bullet &\bullet &\bullet
\end{tikzcd}
\otimes v_{(2,2)},
\]
which corresponds to the simple $\bC_{2,2}$ which we can realize as $\alpha_1 \lhd (1,2)$ where $(1,2)$ is the only simple in $M((1,2))_3$. In degree $5$ the quotient is spanned by
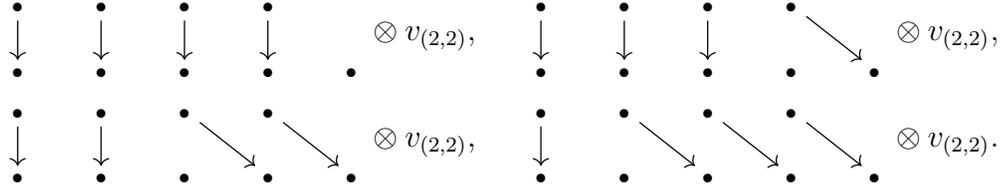
\[{\tiny
\begin{tikzcd}
\bullet \ar[d] & \bullet \ar[d]& \bullet \ar[d]&\bullet \ar[d]&\\
\bullet & \bullet &\bullet &\bullet & \bullet
\end{tikzcd}
\otimes v_{(2,2)}, \qquad
\begin{tikzcd}
\bullet \ar[d] & \bullet \ar[d]& \bullet \ar[d]&\bullet \ar[dr]&\\
\bullet & \bullet &\bullet &\bullet & \bullet
\end{tikzcd}
\otimes v_{(2,2)},}
\]
\[{\tiny
\begin{tikzcd}
\bullet \ar[d] & \bullet \ar[d]& \bullet \ar[dr]&\bullet \ar[dr]&\\
\bullet & \bullet &\bullet &\bullet & \bullet 
\end{tikzcd}
\otimes v_{(2,2)}, \qquad
\begin{tikzcd}
\bullet \ar[d] & \bullet \ar[dr]& \bullet \ar[dr]&\bullet \ar[dr]&\\
\bullet & \bullet &\bullet &\bullet & \bullet 
\end{tikzcd}
\otimes v_{(2,2)}}.
\]
Notice the restriction the quotient places on the remaining basis elements is that the first arrow must go directly down. If it does, the basis element will not be in the quotient, if it does not the basis element is forced into the quotient. The quotient can be viewed as all simples where we place at least one block into the first position of the diagram for $\alpha$, so all the remaining diagrams are exactly the ones where we never have a block in that position. These correspond to the following compositions $\alpha_1 \lhd (1,2,1)$, $\alpha_1 \lhd (1,1,2)$, $\alpha_1 \lhd (1,3)$ and $\alpha_1 + (2,2)$. The composition on the right are exactly those present in $M((1,2))_4$. Notice $\alpha_1$ is split, but in the final diagram we add a box to the first position of $(1,2)$ so we add instead of join.
\end{example}

The above illustrates that the representation theory of $M(\alpha,k)$ is very similar to that of $M(\tau_k(\alpha))$. To make this more explicit,

\begin{lemma} \label{polySimples}
Submodules of $M(\alpha,k)$ are in bijective correspondence with submodules of $M(\tau_k(\alpha))$.
\end{lemma}

\begin{proof}
Simples are indexed by basis elements as we have seen above. We will give a bijection between basis elements that respects the $\cH$-module structure. This bijection then extends to a bijection of submodules, as every submodule can be described by the basis elements that generate it as an $\cH$-module.

Given a basis element of $M(\tau_k(\alpha))_d$, to get the corresponding basis element of $M(\alpha,k)_{d+|\alpha|-k}$ we prepend $|\alpha|-k$ vertices to the top and bottom row of the string diagram corresponding to the basis element and connect them directly downwards. By construction this will be nonzero in the quotient so our map has a nonzero image in $M(\alpha,k)_{d+|\alpha|-k}$. Furthermore, each basis element has a unique image and this map respects the action of $H_{d}(0)$ where we embed $H_d(0) \injects H_{d+|\alpha|-k}(0)$ by $\pi_i \mapsto \pi_{|\alpha|-k+i}$. In particular, we see that if the basis vectors corresponding to simples $\{\bC_{\beta^i}\}$ generate a submodule of $M(\tau_k(\alpha))$, the basis vectors corresponding to the simples $\{\bC_{\gamma^i}\}$ where we prepend the missing part of $\alpha$, i.e. $\eta_{|\alpha|-k}(\alpha)$, to $\beta^i$ will generate a corresponding submodule of $M(\alpha,k)$.

The inverse map is the restriction map that ignores the first $|\alpha|-k$ fixed vertical edges that must occur in any basis element of the quotient module $M(\alpha,k)$. This map is also an injection because by construction every basis element that occurs in any degree of $M(\alpha,k)$ has $|\alpha|-k$ fixed vertical edges, so the only new information is what occurs within the final $d$ pairs of vertices. Notice, this also respects the $H_d(0)$ action.

Finally, by construction the only way to generate new basis elements in $M(\alpha,k)$ from a given element is via the $H_d(0)$ action on the final $d$ vertices. As a result, generating sets for submodules are completely determined up to this action. Since the above bijection respects this $H_d(0)$ action, it gives us a bijection of submodules.
\end{proof}

The above shows that when dealing with proofs about polynomial degree of submodules it suffices to work with the $M(\alpha)$. It is important to note that this does not show that $M(\tau_k(\alpha))$ and $M(\alpha,k)$ are isomorphic as $\cH$-modules, this is very much not the case. It just shows that the simples that occur in submodules of $M(\alpha,k)$ are controlled by $M(\tau_k(\alpha))$ and hence so is the dimension. To understand this Lemma further, look back at Example \ref{quotient1}.

We now prove an important theorem about quotients of $M(\alpha)$.

\begin{theorem} \label{quotients}
Let $\alpha$ be a composition with $|\alpha| = n$. Then any proper quotient of $M(\alpha,k)$ has polynomial degree $< k$. Equivalently, any proper submodule of $M(\alpha,k)$ has Hilbert polynomial of the same degree with the same leading term as $M(\alpha,k)$.
\end{theorem}

Lemma \ref{basis} allows us to place a total order on the basis elements of $M(\alpha)$ for any $\alpha$. In particular, we say that $g_1 \otimes v_\alpha \geq g_2 \otimes v_\alpha$ if $g_1$ has image in $[m]$ while $g_2$ has image in $[n]$ with $m > n$ or if $g_1$ and $g_2$ both have image in $[n]$ and $(g_1(1), g_1(2),\dots, g_1(n)) \geq (g_2(1),g_2(2),\dots, g_2(n))$ in the lexicographic order. It is not hard to see this is indeed a total order.

\begin{lemma} \label{lexPreserve}
Order preserving maps preserve the lexicographic order on basis elements of the same degree.
\end{lemma}

\begin{proof}
Given two basis elements $g_1 \otimes v_\alpha \geq g_2 \otimes v_\alpha$ with $g_1,g_2 \in M(n)_d$ order preserving and another order preserving map $f \in \hom_\cH([d],[m])$ we know that $f$ acts on these basis elements by composition in the first component and we need to show $f (g_1 \otimes v_\alpha) \geq f (g_2 \otimes v_\alpha)$. Equivalently we need to show,
\[
(f g_1) \otimes v_\alpha \geq (f g_2) \otimes v_\alpha.
\]
$f g_1$ and $f g_2$ both inject into $[m]$ so we have to show 
\[
(fg_1(1), fg_1(2),\dots, fg_1(n)) \geq (fg_2(1),fg_2(2),\dots, fg_2(n)).
\]
This is clear because $f$ is order preserving.
\end{proof}

Every element we are concerned with will be a sum of basis elements in the same degree, because we consider degrees separately. The above result implies that order preserving injections between degrees respect the total order on each degree. With these results, we are ready to prove our Theorem,

\begin{proof}[Proof of Theorem \ref{quotients}]
It suffices to prove this result for $M(\alpha) = M(\alpha,n)$ as $\alpha$ varies by Lemma \ref{polySimples}. Let $V$ be a proper submodule of $M(\alpha)$ in ${\rm Mod}_\cH$. Suppose that $d > n$ is the first degree where $V_d \not= 0$. Then because $V$ is a submodule of $M(\alpha)$ we know $V_d$ will contain a vector $v$ with the following vector as a nonzero summand,
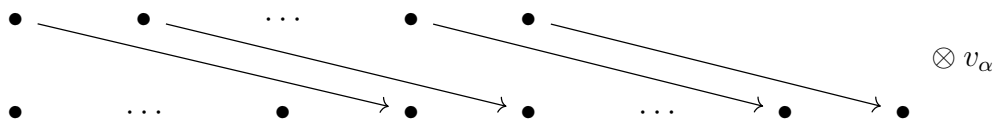
\[
\begin{tikzcd}
\bullet \ar[drrr]& \bullet \ar[drrr] & \cdots &\bullet \ar[drrr]& \bullet \ar[drrr]&&&\\
\bullet&\cdots &\bullet &\bullet& \bullet& \cdots &\bullet &\bullet
\end{tikzcd}
\otimes v_\alpha
\]
where $v_\alpha$ is a vector spanning $\bC_\alpha$ and there are $n$ vertices on top, $d$ on the bottom and the first $d-n$ vertices are not mapped to. We will call the above vector $u$. This vector will appear because we can apply the moves in \eqref{rel1} to obtain such a $v$ having $u$ as a nonzero summand. In degree $d$, $u = f \otimes v_\alpha$ is the largest element with respect to the lex order on basis elements.

If we consider all the order preserving maps that factor through $f$ in later degrees they will be linearly independent because we can see they are basis elements as described in Lemma \ref{basis}. Applying order preserving injections preserves the lex order as seen in Lemma \ref{lexPreserve}, so in any order preserving injection of $v$ the image of $u$ will be the largest vector that appears as a summand. 

Consider a collection of order preserving injections of $v$ that inject $u$ to unique vectors (e.g. we want the final $n$ vertices of the $d$ total vertices to have unique images under each map and we want the first $d-n$ vertices to map directly downward so there are no repetitions in the mapping of $u$). We claim the images of $v$ under these order preserving injections are linearly independent. If there was some nontrivial linear combination of these order preserving injections of $v$, consider the images of $u$. The coefficient of the largest image of $v$, i.e. the one that contains the largest image of $u$ with respect to the lexicographic order, must be zero because the image of $u$ will be the largest possible vector in that degree with respect to the lexicographic order and no other selected order preserving injection of $v$ contains that vector as a summand by construction. We then continue in that fashion to find that all the coefficients must be zero. 

This means that every order preserving injection we have selected of the vector $v$ will be linearly independent. If we map to degree $m \geq d \geq n$ there are $\binom{m-d+n}{n}$ such injections. Each will be given by selecting $n$ vertices from the final $m-d+n$ vertices in degree $m$. So if we consider the submodule generated by the vector $v$, it gives a lower bound on the dimension of the original submodule $V$ in each degree. We know $\dim(M(\alpha)_m)$ gives an upper bound on the dimension of $V_m$. All of these submodules have eventually polynomial growth because they are finitely generated by Noetherianity (Theorem \ref{HNoeth}). In total we have shown that
\[
\binom{m}{n} \geq \dim(V_m) \geq  \binom{m-d+n}{n},
\]
and $\dim(V_m)$ is eventually polynomial. All of this implies that $\dim(V_m)$ must eventually equal a polynomial of degree $n$ with leading coefficient $\frac{1}{n!}$. As a result, the quotient module $M(\alpha)/V$ has degree $< n$ because the leading terms will cancel.
%
%
%
%
%
%
\end{proof}

\section{Grothendieck Group} \label{GrothendieckGroupSec}
Let $\cH_{\leq d} \coloneqq ({\rm Mod}_{\cH})_{\leq d}$ the $\cH$-modules of polynomial degree $\leq d$. Consider the successive Serre quotient category where we consider every polynomial degree $d$ module modulo polynomial degree $< d$ objects, we will denote this category of polynomial degree $d$ objects modulo degree $\leq d-1$ objects by $\cH_{d} \coloneqq \cH_{\leq d}/ \cH_{\leq d-1}$. It turns out that the $M(\alpha,k)$ are simple in the quotient by lower polynomial degree submodules. 

\begin{theorem} \label{smallest}
Any nonzero polynomial degree $k$ quotient of $M(\alpha)$ must contain $M(\alpha,k)$, i.e. $M(\alpha,k)$ is the smallest nonzero polynomial degree $k$ quotient of $M(\alpha)$.
\end{theorem}

\begin{proof}
For any $\alpha$ and $k$, this is equivalent to saying that if $M(\alpha,k) = M(\alpha)/A_{\alpha,k}$ that any other quotient $M(\alpha)/B$ of polynomial degree $k$ has $B \subseteq A_{\alpha,k}$.

If $k = -1$, $M(\alpha,-1) = \bC_\alpha$ which is clearly the smallest polynomial degree $-1$ nonzero quotient of $M(\alpha)$ as any further quotient would have to be identically zero. 

Assume $k \geq 0$. Suppose there were a submodule $B$ with $A_{\alpha,k} \subset B$ such that $M(\alpha)/B$ had polynomial degree $k$. This implies the existence of a nonzero submodule of $M(\alpha,k)$ such that the quotient still has polynomial degree $k$ which is impossible by Theorem \ref{quotients}, as this shows the polynomial degree must drop to $k-1$.

The other option is that $A_{\alpha,k}$ and $B$ are not comparable in the containment order. Suppose this is the case with $k \geq 0$, and that $M(\alpha)/B$ has polynomial degree $k$. 

Let $b$ span a simple in $B$ that is not contained in $A_{\alpha,k}$. For $b$ to not be in $A_{\alpha,k}$, the corresponding simple module must not have a box added to the first $|\alpha|-k$ positions of $\alpha$. This is because by construction $A_{\alpha,k}$ contains every such simple module. Suppose $b$ is in $\cH$-degree $\ell$. Without loss we can assume we add all $\ell$ boxes to the $(|\alpha|-k+1)$-st position since the submodule generated by any other choice of simple contains this one because we can always shift boxes down. 

In the quotient of $M(\alpha)$ by the submodule generated by $b$, the remaining simples correspond to diagrams where we add no more than $\ell-1$ boxes to the first $|\alpha|-k+1$ positions of $\alpha$. The number of ways to place $\ell-1$ boxes in the first $|\alpha|-k+1$ positions of $\alpha$ is $\binom{|\alpha|-k+1+\ell-2}{\ell-2}$ which is a constant. There are a total of $|\alpha|+1$ positions we could add boxes to, and once we fix a way to place at most $\ell-1$ boxes in the first $|\alpha|-k+1$ positions, we can add the remaining boxes in any way to the remaining $k$ positions. In degree $d$, we could choose to add at most $d-|\alpha|$ boxes to the top $k$ positions, there are at most $\binom{d-|\alpha|+k-1}{k-1}$ ways to do this. The quotient then has dimension $\leq \binom{|\alpha|-k+1+\ell-2}{\ell-2} \binom{d-|\alpha|+k-1}{k-1}$, this is a degree $k-1$ polynomial in $d$. $B$ clearly contains the submodule generated by $b$ and so this implies the quotient $M(\alpha)/B$ has smaller polynomial degree and hence polynomial degree $\leq k-1$ which is a contradiction. \qedhere
\end{proof}

%

\begin{lemma} \label{simple}
The module $M(\alpha,k)$ is finitely generated and simple in $\cH_{k}$. Furthermore, every simple module is of this form.
\end{lemma}

\begin{proof}
Finite generation is inherited in the quotient category. To see that $M(\alpha,k)$ is simple, Theorem \ref{quotients} implies that if $V$ is any nonzero submodule it must have the same polynomial degree and leading coefficient as $M(\alpha,k)$ and so the quotient has smaller degree. This implies that $M(\alpha,k) = V$ in the quotient so $M(\alpha,k)$ is simple.

For the final statement, let $L \in \cH_{k}$ be simple. Denote by $\tilde L$ a lift of $L$ to an object in ${\rm Mod}_{\cH}$. We may assume that $\tilde{L}$ is a quotient of $M(\beta)$ for some choice of $\beta$. Any simple quotient of $M(\beta)$ must be isomorphic to $M(\beta,k)$ for some $k$ in $\cH_{k}$ after applying the localization functor. Indeed, Theorem \ref{smallest} shows $M(\beta,k)$ is the smallest degree $k$ quotient of $M(\beta)$, so $M(\beta,k)$ is a submodule of the quotient, but because the quotient is simple they must be equal.
\end{proof}

We define ${\rm Mod}_K = \bigoplus_{k \geq 0} \cH_{k}$. That is, we consider all polynomial degree $d$ objects modulo lower degree objects.

\begin{lemma} \label{finite length}
Every object of ${\rm Mod}_{K}$ has finite length.
\end{lemma}

\begin{proof}
Every module $V$ is a quotient of a finite direct sum of $M(\alpha)$ and Lemma \ref{simple} implies that the images of the $M(\alpha)$ are simple and so clearly of finite length in ${\rm Mod}_{K}$, hence $V$ is finite length.
\end{proof}

\begin{theorem} \label{decomp}
Every $\cH$-module $V$ with polynomial degree $d$ has a finite filtration in $\cH_{d}$
\[
0 = F_1 \subset F_2 \subset \cdots \subset F_n = V,
\]
where $F_i/F_{i-1}$ is isomorphic to $M(\alpha_i,d)$.
\end{theorem}

\begin{proof}
$V$ has a finite composition series from Lemma \ref{finite length}, because $V$ has polynomial degree $d$ we know the simples must as well otherwise we would be able to shorten the composition series. This implies that $F_i/F_{i-1}$ is isomorphic to a simple module of degree $d$, so a simple module in $\cH_{d}$. Lemma \ref{simple} implies each of the simples are isomorphic to $M(\alpha_i,d)$.
\end{proof}

\begin{theorem} \label{span}
The isomorphism classes, $[M(\alpha,k)]$ span $\cG({\rm Mod}_\cH)$.
\end{theorem}

\begin{proof}
Given any finitely generated $\cH$-module $V$ of polynomial growth $d$, we claim that its isomorphism class can be expressed as a finite sum
\[
[V] = \sum_{j} [M(\alpha_j,k_j)]
\]
where $k_j \leq d$. We will proceed by induction on the polynomial growth of $V$. If $d = -1$, $V$ is torsion and even without using Theorem \ref{decomp} we can express $V$ as a sum
\[
[V] = \sum_i [M(\alpha_i,-1)]
\]
by starting in the highest $\cH$-degree and injecting the corresponding torsion modules. Now suppose the result holds for degrees $\leq d-1$ and consider some $d \geq 0$. From Theorem \ref{decomp} we can write 
\[
[V] = \sum_{i} [M(\alpha_i,d)] + [W]
\]
where $[W]$ is a module with smaller polynomial degree $\leq d-1$ and the first sum is finite. By induction we have
\[
[W] = \sum_{j} [M(\alpha_j,k_j)]
\]
where $k_j \leq d-1$. This completes the proof. 
\end{proof}

We now wish to strengthen this result to show that the $[M(\alpha,k)]$ actually form a basis.

\begin{theorem} \label{GrothBasis}
The $[M(\alpha,k)]$ form a basis for $\cG({\rm Mod}_\cH)$.
\end{theorem}

\begin{proof}
Theorem \ref{span} implies the $[M(\alpha,k)]$ span $\cG({\rm Mod}_\cH)$, it remains to prove they are linearly independent. To do this, we will construct functions.

To prove linear independence, it suffices to prove that $M(\alpha,k)$ are linearly independent for a fixed $k \geq -1$. This follows from considering the corresponding Hilbert polynomials. If there were a finite linear combination,
\[
\sum_{i,j} c_{i,j} [M(\alpha_i,j)] = 0,
\]
if we consider the corresponding sum of Hilbert polynomials for sufficiently large $\cH$-degree, that sum must be zero as well. This implies that the coefficients of each polynomial degree must be zero. By considering the largest degree, say $\ell$, such coefficients can only appear from the $[M(\alpha_i,\ell)]$. If we can show linear independence here, it implies the above sum can be zero if and only if the coefficients $c_{i,\ell} = 0$. We can continue in this fashion to show that every coefficient must be zero. So to show linear independence in general, it suffices to prove it for each fixed $k \geq -1$. 

Fix some $k \geq -1$. For a finitely generated $\cH$-module $V$, let $\Gamma_\alpha([V])$ be the number of times $\bC_\alpha$ appears in the composition series of the $V_{|\alpha|}$. This clearly respects short exact sequences and so is well defined on the Grothendieck group. Notice that for a fixed $k \geq -1$, $\Gamma_\alpha([M(\beta,k)])$ satisfies
\[
\Gamma_\alpha([M(\beta,k)]) =
\begin{cases}
0 & |\beta| > |\alpha|, \; \text{or} \; |\alpha| = |\beta| \; \text{and} \; \alpha\not= \beta\\
1 & \beta = \alpha\\
\ast & \text{else}
\end{cases}
\]
where $\ast$ could be any non-negative integer. First, if $|\beta| > |\alpha|$ the lowest nonzero $\cH$-degree of $M(\beta,k)$ is $|\beta|$ so $\bC_\alpha$ could not possibly occur because it only appears in $\cH$-degree $|\alpha|$. If $|\alpha| = |\beta|$, there are two cases to consider. If $\alpha = \beta$, then $\bC_\alpha$ appears in the lowest degree of $M(\alpha,k)$ precisely once, so $\Gamma_\alpha(M(\alpha,k)) = 1$. If $\alpha \not= \beta$, then $M(\beta,k)$ has only $\bC_\beta$ in degree $|\beta| = |\alpha|$ and so $\bC_\alpha$ also cannot occur. It is possible that $\bC_\alpha$ could appear multiple times in some $M(\beta,k)$ with $|\beta| < |\alpha|$, this is the remaining case, but this will not matter.

Order the compositions of $i$ arbitrarily, then place all compositions in increasing order according to $i$. Call this list of all compositions of any non-negative integer $\{\alpha_i\}$. The order described is equivalent to saying that if $i \leq j$ we have $|\alpha_i| \leq |\alpha_j|$ so the size of the compositions is weakly increasing in our list. Consider the matrix whose $(i,j)$ entry is $\Gamma_{\alpha_i}(M(\beta_j,k))$. The above implies that this matrix will be lower uni-triangular because $\Gamma_{\alpha_i}(M(\alpha_j,k))$ is potentially nonzero precisely when $j \leq i$ as here we could have $|\alpha_j| < |\alpha_i|$, and when $j = i$ $\Gamma_{\alpha_i}(M(\alpha_i,k)) = 1$. When $j > i$, we either have $|\alpha_i| = |\alpha_j|$ but $\alpha_i \not= \alpha_j$ in which case the $(i,j)$-entry is zero, or $|\alpha_j| > |\alpha_i|$ in which case again the $(i,j)$-entry is zero. 

This implies that the matrix is invertible and so there is some change of basis in which we can express the Kronecker-Delta functions $\delta_{\alpha,k}$ as a linear combination of the $\Gamma_{\alpha}$. Where recall that 
\[
\delta_{\alpha,k}(M(\beta,k)) \coloneqq
\begin{cases}
1 & (\alpha,k) = (\beta,k)\\
0 & \text{else}.
\end{cases}
\]
The existence of such functions implies that the $[M(\alpha,k)]$ are linearly independent for a fixed $k \geq -1$. Indeed, suppose there were some finite linear combination
\[
\sum_{i} c_{\alpha_i} [M(\alpha_i,k)] = 0,
\]
then applying $\delta_{\alpha,k}$ implies $c_{\alpha,k} = 0$. As this holds for every choice of $\alpha$, this implies every coefficient is zero and so they must be linearly independent. This argument holds for any fixed $k \geq -1$ and so implies that the $[M(\alpha,k)]$ are linearly independent for any such fixed $k$. By the above, we have that the $[M(\alpha,k)]$ are linearly independent as $\alpha$ ranges over all compositions and $k$ an integer with $k\geq -1$.
\end{proof}

Lemma \ref{simple} describes the simple modules in $\cH_d = \cH_{\leq d}/\cH_{\leq d-1}$, for the coming results we wish to strengthen this slightly. 

\begin{lemma} \label{simpleH}
Any simple object in $\cH/\cH_{\leq d-1}$ is isomorphic to $M(\alpha,d)$ for some choice of $\alpha$.
\end{lemma}

\begin{proof}
Any simple object in $\cH_{\leq d}/\cH_{\leq d-1}$ is isomorphic to $M(\alpha,d)$ by Lemma \ref{simple}. So the $M(\alpha,d)$ are simple, it remains to consider higher polynomial degree modules. Given any other simple object $L$ that has polynomial degree $> d$, we can lift $L$ to $\tilde{L}$ in ${\rm Mod}_\cH$. As in Lemma \ref{simple}, we may assume that $\tilde{L}$ is a quotient of $M(\beta)$ for some choice of $\beta$. Any quotient of $M(\beta)$ must contain $M(\beta,k)$ for an appropriate choice of $k > d$ by Theorem \ref{quotients}. Hence it cannot possibly be simple unless it equals $M(\beta,k)$ in the localization. But $M(\beta,k)$ is only simple in the quotient by $\cH_{\leq k-1}$, otherwise it has a nontrivial quotient corresponding to $M(\beta,k-1)$ which implies the existence of a nontrivial submodule of polynomial degree $\geq d-1$. This means that the only simples in $\cH/\cH_{\leq d-1}$ are exactly $M(\beta,d)$ as $\beta$ varies over all compositions.
\end{proof}

\begin{proposition} \label{degreeEqualsDim}
A finitely generated $\cH$-module has polynomial degree $\leq d$ if and only if it has Gabriel-Krull dimension $\leq d$.
\end{proposition}

\begin{proof}
Let $\cH^{\leq d}$ be the objects that are finite length in the quotient $\cH/\cH^{\leq d-1}$. We will prove by induction on $d$ that $\cH^{\leq d} = \cH_{\leq d}$. 

For $d = -1$ the statement is true because in each case we get torsion modules. These are the only polynomial degree $-1$ modules by definition and they are also the only finite length modules in $\cH$ by Lemma \ref{HSimples}. 

Now if we assume the statement for $\leq d-1$, we will prove it for $d$ where $d \geq 0$. We aim to show that a finitely generated $\cH$-module $V$ has polynomial degree $\leq d$ if and only if the image of $V$ in $\cH/\cH^{\leq d-1}$ is finite length. By induction this is equivalent to showing that $V$ has polynomial degree $\leq d$ if and only if the image of $V$ in $\cH/\cH_{\leq d-1}$ is finite length because we have $\cH_{\leq d-1} = \cH^{\leq d-1}$.

If $V$ has polynomial degree $\leq d$ then it is is finite length in $\cH/\cH_{\leq d-1}$ by Lemma \ref{finite length}. 

Suppose instead that $V$ is finite length in $\cH/\cH_{\leq d-1}$. Consider the finite composition series for $V$ in $\cH/\cH_{\leq d}$. Lemma \ref{simpleH} implies $[V]$ is a sum of $[M(\alpha_i,d)]$ modulo lower polynomial degree terms as these are precisely the simple modules in $\cH/\cH_{\leq d}$. This means $V$ has polynomial degree $\leq d$ as each of the terms in the expression of $[V]$ in the Grothendieck group does. The result then follows by induction.
\end{proof}

\begin{corollary}
The Gabriel-Krull dimension of $\cH$ is $\infty$.
\end{corollary}

\begin{proof}
This follows immediately from Proposition \ref{degreeEqualsDim} and the fact that we have $\cH$-modules with arbitrarily large polynomial degree, e.g. $M((d))$ is a degree $d$ $\cH$-module.
\end{proof}

Heuristically, $\cH$ can be thought of as somewhere between ${\bf OI}$ and ${\bf FI}$ but closer to ${\bf OI}$. In particular, we know that ${\bf OI}$ is a subcategory of ${\cH}$ under the map $\Phi \colon {\bf OI} \to {\cH}$ described in Proposition \ref{OIsub}. This gives rise to a pullback functor $\Phi^\ast \colon {\rm Mod}_{\cH} \to {\rm Mod}_{{\bf OI}}$ which then induces a pullback functor on Grothendieck groups which by abuse of notation we will also call $\Phi^\ast$. It is natural to ask where the $[M(\alpha,k)]$ land. 

We will use the notation outlined in \cite{GS}. In their paper, G{\"u}nt{\"u}rk{\"u}n and Snowden study the representation theory of the increasing monoid, denoted $\cJ$, which they show also describes the representation theory of ${\bf OI}$-modules. We recall the essential definitions from their paper, for further details consult \cite{GS}.

One of the most important classes of $\cJ$-modules are called {\em standard modules}, denoted by $E^\lambda$ where $\lambda$ is a finite word on the alphabet $\{a,b\}$. Given such a word $\lambda = \lambda_1\dots\lambda_r$ the module $E^\lambda$ associated to $\lambda$ is defined as follows: $E^\lambda$ has a basis consisting of elements of the form $e_{i_1,\dots,i_r}$ where $1 \leq i_1 < \cdots < i_r$ subject to the constraint that if $\lambda_k = b$ then $i_k - i_{k-1} = 1$ (with the convention that $i_0 = 0$). The increasing monoid acts by $\sigma e_{i_1,\dots,i_r} = e_{\sigma(i_1),\dots, \sigma(i_r)}$ if this is a basis element of $E^\lambda$ and $0$ otherwise.

In \cite[Theorem 12.1]{GS} the authors show that the isomorphism classes of standard modules form an integral basis for the Grothendieck group of finitely generated $\cJ$-modules. To understand $\Phi^\ast([M(\alpha,k)])$ it thus suffices to express the image in terms of the isomorphism classes of these standard modules.

\begin{proposition} \label{pullback}
Under the pullback functor we have $\Phi^\ast([M(\alpha,k)]) = [\ul{E}^{b^{|\alpha|-k}a^{k+1}}]$.
\end{proposition}

\begin{proof}
The pullback functor forgets the action of $H_n(0)$ in each $\cH$-degree and only allows order preserving injections between degrees. This means that when we consider basis elements for $M(\alpha,k)$ indexed by order preserving injections, we can forget the vector $v_\alpha$ that spans $\bC_\alpha$ as this only keeps track of the 0-Hecke action in each degree. We then see that in each degree $d$, a basis for $M(\alpha,k)$ is given exactly by order preserving injections $g \colon [|\alpha|] \to [d]$ where the first $|\alpha|-k$ numbers must map to themselves. Such order preserving injections are in bijective correspondence with basis elements of $\ul{E}^{b^{|\alpha|-k}a^{k+1}}$ where we send the injection $g \colon [|\alpha|] \to [d]$ to the basis element $e_{g(1),\dots, g(|\alpha|), d+1}$. Notice by construction $g(i)=i$ for $i = 1,\dots,|\alpha|-k$ which is exactly mandated by the constraint word $b^{|\alpha|-k}a^{k+1}$.

This map also respects the transition maps between degrees because if such a map does not fix the first $|\alpha|-k$ entries, it will be zero in both $M(\alpha,k)$ and $\ul{E}^{b^{|\alpha|-k}a^{k+1}}$. Otherwise, the remaining entries are free to be sent anywhere in both modules.
\end{proof}

We show in Theorem \ref{GrothBasis} that the $[M(\alpha,k)]$ form a basis for the Grothendieck group of $\cH$. Proposition \ref{pullback} shows that $\Phi^\ast([M(\alpha,k)]) = \Phi^\ast([M(\beta,k)])$ so long as $|\alpha| = |\beta|$. Heuristically, this is because the only difference between these two modules is the 0-Hecke action in each degree. When we forget this action, they become the same ${\bf OI}$-module. 

The fact that $\Phi^\ast$ is not surjective can be explained intuitively by $\cH$-modules allowing us to move basis vector upwards within degrees. The basis elements we miss are the $[E^\lambda]$ where $\lambda$ has $b$ elements interspersed within $a$ elements. This would correspond to an order preserving injection where we force $g(k) = g(k-1)+1$. This condition cannot be preserved by a $\cH$-action unless we only fix the beginning elements. This is illustrative of the statement that every $\cH$-module is an ${\bf OI}$-module but not every ${\bf OI}$-module is a $\cH$-module. Also, it illustrates that there is a difference between the categories ${\bf OI}$ and $\cH$.


\section{Representation Stability}
With the previous section established, we are ready to state some immediate representation stability results for sequences of representations of $H_n(0)$-modules.

\begin{theorem} \label{repStability}
For any finitely generated $\cH$-module $V$, we have a unique finite decomposition in the Grothendieck group $\cG({\rm Mod}_\cH)$,
\[
[V] = \sum_{i,j} c_{\alpha^i,k_j} [M(\alpha^i,k_j)]
\]
where the coefficients $c_{\alpha^i,k_j}$ are integers.
\end{theorem}

\begin{proof}
This follows immediately from Theorem \ref{GrothBasis}.
\end{proof}

We will say that a sequence of $H_n(0)$ representations is {\em representation stable} if there is such a decomposition in the Grothendieck group. The above shows that being finitely generated as a $\cH$-module implies representation stability. We can make this more explicit in terms of the simple modules that can occur in finitely generated $\cH$-modules for sufficiently large $n$.

\begin{theorem}
For any finitely generated $\cH$-module $V$, for sufficiently large $n$ the simples that appear in $V_n$ are exactly those that appear in the finite sum,
\[
\bigoplus_{i,j} c_{\alpha^i,k_j} M(\alpha^i,k_j)_n
\]
where the coefficient $c_{\alpha^i,k_j}$ is independent of $n$.
\end{theorem}

\begin{proof}
This is a direct consequence of \ref{repStability}.
\end{proof}

\begin{remark} \label{repStabilitySimples}
For any finitely generated $\cH$-module $V$, we have for sufficiently large $n$ a unique finite list of pairs $\{(\alpha^i,k_i)\}$ of compositions with integers $k \geq -1$ that dictate exactly which simples occur in $V_n$.
\end{remark}

To see what we mean in this theorem, notice that in $M(\alpha,k)$ only $\bC_\alpha$ occurs in degree $|\alpha|$. After this, to find the simples that occur in degree $d \geq |\alpha|$ you add $d-|\alpha|$ boxes to the top $k+1$ positions of $\alpha$ in all ways possible. Consider the following example,

\begin{example}
Consider the module $M((2,1),1)$. The first nonzero degree is $3 = |\alpha|$ and here we have the single simple corresponding to $\young(:~,~~)$. We will compute the next few degrees and denote the boxes we add to the top $2$ positions by ${\tiny \young(\bullet)}$ .
\begin{align*}
M((2,1),1)_3 &= \young(:~,~~)\\
M((2,1),1)_4 &= \young(:\bullet,:~,~~) + \young(:\bullet~,~~:) \\
M((2,1),1)_5 &= \young(:\bullet\bullet,:~:,~~:) + \young(:\bullet\bullet~,~~::) + \young(::\bullet,:\bullet~,~~:)\\
M((2,1),1)_6 &= \young(:\bullet\bullet\bullet,:~::,~~::) + \young(:\bullet\bullet\bullet~,~~:::) + \young(::\bullet\bullet,:\bullet~:,~~::) + \young(:::\bullet,:\bullet\bullet~,~~::) \;.
\end{align*}
So if we can write a finitely generated $\cH$-module as $[V] = [M((2,1),1)] + [M((2),1)]$ then for sufficiently large $n$ the diagrams that occur result from adding boxes to the two $2$ positions of the ribbon diagram corresponding to $(2,1)$ and adding boxes to the top $2$ positions of the ribbon diagram $(2)$. In general we add boxes to the top $k+1$ positions of the ribbon diagram corresponding to $\alpha$.
\end{example}

\section{The Shift Functor}
As in the case of ${\bf FI}$ it makes sense to define a shift functor $\Sigma$ for $\cH$. Given a $\cH$-module $V$, we define its {\bf first shift} $\Sigma V$ to be the $\cH$-module with $(\Sigma V)_n = V_{n+1}$ on sets. For any $f \in \hom_{\cH}([n],[m])$, the map $\Sigma V(f) \colon V_{n+1} \to V_{m+1}$ is the map $V(F)$ where $F$ agrees with $f$ on $[n]$ and maps $n+1$ to $m+1$. We define by $\Sigma_a$ the $a$th iterate of $\Sigma$. From this definition it is not hard to see that
\[
(\Sigma V)_n \cong {\rm Res}^{H_{n+1}(0)}_{H_n(0)} V_{n+1}.
\]
Also, notice there is an inclusion map $V \to \Sigma V$ induced by the natural inclusion $\iota_n \colon [n] \to [n+1]$ with $\iota_n(i) = i$ for all $i$. We define the {\bf derivative} of $V$, denoted by $\Delta V$ as
\[
\Delta V \coloneqq \coker(V \to \Sigma V).
\]
In the case of {\bf FI} and many other categories of interest these functors play an important role and have many nice properties. For further reading on this subject we refer the reader to \cite{Ram, CEF, CEFN}. We will see that many of these desirable properties are also present for $\cH$.

\begin{lemma} \label{shift}
Given any padded induced module $M(\alpha,k)$ with $k \geq 1$ we have
\[
\Sigma M(\alpha,k) = M(\alpha,k) \oplus M(\alpha_{|\alpha|-1},k-1).
\]
If $k = 0$, $\Sigma M(\alpha,0) = M(\alpha,0)$. If $k = -1$, $\Sigma M(\alpha,-1) = 0$.
\end{lemma}

\begin{proof}
To prove this, fix an $\cH$-degree $d$. In this degree we know $(\Sigma M(\alpha,k))_d = M(\alpha,k)_{d+1}$ as a set, where the action comes from restriction. Take the basis for $M(\alpha,k)$ as in Lemma \ref{basis}. In $M(\alpha,k)_{d+1}$ we can separate the basis elements into two groups, the first consisting of all elements indexed by an order preserving injection with $g(|\alpha|) = d+1$ and the second consisting of the remaining elements where $g(|\alpha|) \leq d$. In ${\rm Res}_{H_d(0)}^{H_{d+1}(0)} M(\alpha,k)_{d+1}$ the first group becomes a sub-representation isomorphic to $M(\alpha_{|\alpha|-1},k-1)_d$ and the second becomes a sub-representation isomorphic to $M(\alpha,k)_d$. This is clear because when we restrict we ignore $\pi_{d}$. Under the action of $\pi_1,\dots,\pi_{d-1}$ the first collection of elements consists of all order preserving injections from $[|\alpha|-1]$ to $[d]$ and the action of $\pi_1,\dots,\pi_{d-1}$ is dictated by $\bC_{\alpha_{|\alpha|-1}}$. The second collection of basis elements consists of all order preserving injections from $[|\alpha|]$ to $[d]$ and the action of $\pi_1,\dots,\pi_{d-1}$ is dictated by $\bC_{\alpha}$. 

Furthermore, because in $\Sigma M(\alpha,k)$ we must map the final vertex to the final vertex these groupings of basis elements are preserved under all transition maps, so they do form $\cH$-submodules. The above shows that these submodules are isomorphic to $M(\alpha,k)$ and $M(\alpha_{|\alpha|-1},k-1)$. It is also clear from the above that we have the short exact sequence
\[
0 \to M(\alpha_{|\alpha|-1},k-1)\to \Sigma M(\alpha,k) \to M(\alpha,k) \to 0
\]
and that it splits, where the splitting is exactly the inclusion of $M(\alpha,k)$ into $\Sigma M(\alpha,k)$.
\end{proof}

\section{Regularity}
One can also define a notion of regularity for $\cH$-modules. Let $V$ be a $\cH$-module. We define ${\rm Tor}_0(M)$ to be the $\cH$-module that assigns to the set $S$ the quotient of $V(S)$ by the sum of the images of the $V(T)$ as $T$ varies over all proper subsets of $S$. This is analogous with $H_0(V)$ the $0$-th homology. It is easy to see that ${\rm Tor}_0$ is a right-exact functor, so we consider its left derived functors ${\rm Tor}_\bullet$. We will also sometimes refer to these as $H_\bullet$ the $\cH$-homology. We let $t_i(V)$ be the maximum $\cH$-degree occurring in ${\rm Tor}_i(V)$. We use the convention $t_i(V) = -1$ if ${\rm Tor}_i(V) = 0$. We define the {\em regularity} of $V$, denoted ${\rm reg}(M)$, to be the minimum integer $r$ such that $t_i(V) \leq r + i$ for all $i$. We will show that every finitely generated $\cH$-module has finite regularity.

To do this we will use the following fact.

\begin{theorem}[\cite{GL}] \label{OIFinReg}
Every finitely generated ${\bf OI}$-module has finite regularity.
\end{theorem}

As seen in \S\ref{NotationSec}, we have an essentially surjective inclusion $\Phi \colon {\bf OI} \to \cH$. We used this map to show $\cH$ is a noetherian category. We can also use it in combination with the previous theorem to show that finitely generated $\cH$-modules have finite regularity. For the following proposition we will assume all of our categories are $\bk$-linear directed categories whose set of isomorphism classes of objects is identified with $\bZ_{\geq 0}$ as a poset.

\begin{proposition} \label{regularityTransfer}
Let $\Phi \colon \cC \to \cC'$ be an essentially surjective functor such that for every positive-degree morphism $f$ in $\cC'$ there exists a morphism $\sigma$ in $\cC'$ such that $f = \Phi(g) \sigma$. For any $\bC'$-module $V$ and $i \geq 0$, we have
\[
\Phi^\ast({\rm Tor}_i^{\cC'}(V)) = {\rm Tor}_i^\cC(\Phi^\ast(V)).
\]
\end{proposition}

\begin{proof}
We will first prove this for $i = 0$. The left hand side is then the quotient of $V$ by all the elements of the form $f(x)$ where $f$ is a positive-degree morphism in $\cC'$ and $x$ is an element of $V$. By our assumption for each such $f$ there is a $\sigma$ such that $f = \Phi(g) \sigma$ for some positive-degree morphism in $\cC$. So quotienting by $f(x)$ is the same as quotienting out by $g(\sigma(x))$. This proves the case $i = 0$. 

For $i > 0$, we note that $\Phi^\ast$ is exact so both sides of the equality are left derived functors. Since they agree in degree $0$, they are the same.
\end{proof}

\begin{remark}
We thank Andrew Snowden and Steven Sam for sharing an early copy of their book which contains Proposition \ref{regularityTransfer}.
\end{remark}

\begin{theorem}
Every finitely generated $\cH$-module $V$ has finite regularity.
\end{theorem}

\begin{proof}
Let $\Phi \colon {\bf OI} \to \cH$ be the natural inclusion. Note that for any map $f \in \hom_{\cH}([n],[m])$ we can express $f$ as an element $\sigma \in H_n(0)$ followed by an order preserving injection $\iota$ from $[n]$ to $[m]$. That is $f = \Phi(\iota) \sigma$, so $\Phi$ satisfies the necessary conditions in Proposition \ref{regularityTransfer}. Since $V$ is finitely generated as a $\cH$-module, this Proposition implies that it will also be finitely generated as an ${\bf OI}$-module. This is because finite generation means that $t_0(V)$ is bounded, i.e. that ${\rm Tor}_i^{\cC'}(V)_d = 0$ for $d \gg 0$. Proposition \ref{regularityTransfer} then implies that ${\rm Tor}_i^{\cC}(\Phi^\ast(V))_d = 0$ for $d \gg 0$. 

But then Theorem \ref{OIFinReg} shows that $\Phi^\ast(V)$ has finite regularity. This in turn implies that $V$ must also have finite regularity as a $\cH$-module.
\end{proof}

\begin{corollary} \label{OIFinGen}
A $\cH$-module $V$ is finitely generated if and only if it is finitely generated as an ${\bf OI}$-module.
\end{corollary}

\begin{remark}
The same is true of ${\bf FI}$-modules. It is important to remember that being a $\cH$-module (or an {\bf FI}-module) carries more information than being a ${\bf OI}$-module. The interesting part of the previous corollary is that to prove finite generation one only needs to consider the ${\bf OI}$ action. We also note that the generation degree may be different.
\end{remark}

\section{Degree $d$ polynomials in $n$ variables} \label{polynomialSec}
The first example we consider is probably the simplest. Consider the $\cH$-module $V$ with $V_n = \bk[x_1,\dots,x_n]_d$, that is degree $d$ polynomials in $n$ variables. There is a $H_n(0)$ action on $V_n$ via {\bf Demazure operations}, in particular the element $\tau_i$ of $H_n(0)$ acts by the operator $\pi_i$ where
\[
\pi_i(f) \coloneqq \partial_i(x_i f) = \frac{x_if - s_i(x_if)}{x_i-x_{i+1}},
\]
where $\partial_i$ is the {\bf divided difference operator} and $s_i$ is the transposition $(i,i+1)$ acting on indices.

Given a map $f \in \hom_{\cH}([n],[m])$, we can factor $f$ as a series of $\pi_i$ in $H_n(0)$ followed by the injection of variables $\iota_{n,m}$ followed by a series of $\pi_i$ in $H_m(0)$. $f$ acts as the Demazure operators corresponding to the $\pi_i$ followed by the injection of variables $\iota_{n,m}$ then the other Demazure operators corresponding to the remaining $\pi_i$ in $H_m(0)$.

For example, consider the map in $\hom_{\cH}([3],[5])$ corresponding to the diagram
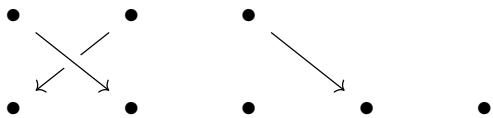
\[
\begin{tikzcd}
\bullet & \bullet \arrow[dl]& \bullet \ar[dr] &&\\
\bullet &\bullet \ar[from = ul, crossing over]& \bullet&\bullet&\bullet\\
\end{tikzcd}
\]
we can factor this as the map $\pi_1$ followed by $\iota_{3,5}$ followed by $\pi_3$. When we apply this to the element $x_1$ in the case where $d = 1$ we get $x_1+x_2$ because $\pi_1(x_1) = x_1 + x_2$ and $\pi_3(x_1 + x_2) = x_1+x_2$.

Consider the more complicated element $x_2 x_3$ in the case $d = 2$. In this case when we apply $\pi_1$ we get $0$. Then when we apply $\pi_3$ we still have $0$.

We can use our criterion \ref{criterion} to verify this is a $\cH$-module. In this case the natural inclusion map will be an inclusion of variables. We just have to verify that for the natural inclusion of variables $\bk[x_1,\dots,x_n] \to \bk[x_1,\dots,x_m]$ the Demazure operators $\pi_{n+1},\dots,\pi_{m}$ act by 1. This is clearly the case since
\[
\pi_{n+1}(f(x_1,\dots,x_n,0,\dots,0)) = \frac{x_{n+1} f(x_1,\dots,x_n,0,\dots,0) - x_{n+2} f(x_1,\dots,x_n,0,\dots,0)}{x_{n+1} - x_{n+2}} = f.
\]

Notice that $V$ is finitely generated as a $\cH$-module in degrees $\leq d$. This is because once we have $d$ variables, we can get every composition of $d$, so we have all unique exponent vectors of monomials. We then have to inject to higher degrees to get all other monomials. Theorem \ref{repStability} implies we should be able to decompose $V$ in the Grothendieck group $\cG({\rm Mod}_K)$. We will show how to do this explicitly in the case that $d = 1$ and $d=2$. 

When $d = 1$ we have $V_1 = {\rm Span}_{\bF}(x_1)$, $V_2 = {\rm Span}_{\bF}(x_1,x_2)$ etc. We will now explicitly decompose these $V_i$ in the Grothendieck group. We will do this for $i = 2$ because it illustrates the process in general. $V_2$ has a one dimensional subspace spanned by $x_2$ where $\pi_1(x_2) = 0$. This corresponds to $\bC_{(1,1)}$. When we quotient by this space $\pi_1(x_1) = x_1$ so in the quotient $x_1$ spans a subspace isomorphic to $\bC_{(2)}$. Hence we have
\[
[V]_2 = \young(~,~) + \young(~~) \;.
\]
If we continue to decompose in a similar fashion we find
\begin{align*}
[V]_1 &= \young(~)\;.\\
[V]_2 &= \young(~~) + \young(~,~)\; . \\
[V]_3 &= \young(:~,~~) + \young(~~,~:) + \young(~~~) \;.\\
[V]_4 &= \young(::~,~~~) + \young(:~~,~~:) + \young(~~~,~::) + \young(~~~~) \;.
\end{align*}
It is actually not hard to see that $V = M({\tiny \young(~)})$. According to Theorem \ref{repStability}, in $\cG({\rm Mod}_\cH)$ we have uniquely that,
\[
[V_n] = [M({\tiny \tiny \tiny\young(~)})_n].
\]
When $d = 2$, we have $V_1 = {\rm Span}_{\bF}(x_1^2)$, $V_2 = {\rm Span}_{\bF}(x_1^2,x_2^2,x_1x_2)$, etc. The $\cH$ action is described by injecting variables $x_i \mapsto x_i$ and applying the corresponding Demazure operators. For example take $f \in \hom_{\cH}([2],[3])$,
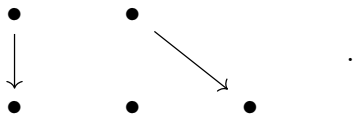
\[
\begin{tikzcd}
\bullet \ar[d] & \bullet \ar[dr]& &\\
\bullet & \bullet & \bullet &
\end{tikzcd}.
\]
The corresponding map $f_\ast(x_1^2) = x_1^2$ because it will correspond to $\pi_2(x_1^2)$ because the map $f$ factors as $\pi_2 \iota_{2,3}$. $f_\ast(x_2^2) = \pi_2(x_2^2) = x_2^2 + x_2x_3 + x_3^2$. Finally $f_\ast(x_1x_2) = x_1x_2 + x_1 x_3$. From this example it is not hard to see this $\cH$ module is finitely generated in degree $2$. It is not generated in degree $1$ because the images of $x_1^2$ are only $x_1^2$ and $x_1^2 + x_1x_2 + x_2^2$. This is also expected from an easy dimension count.

If we carry out the decomposition in the Grothendieck group, we find that
\begin{align*}
[V]_1 &= \young(~)\;.\\ 
[V]_2 &= 2 \; \young(~~) + \young(~,~)\; . \\
[V]_3 &= 2 \; \young(:~,~~) + 2 \; \young(~~,~:) + 2 \; \young(~~~) \;.\\
[V]_4 &= 2 \; \young(::~,~~~) + 3 \; \young(:~~,~~:) + 2 \; \young(~~~,~::) + 2 \; \young(~~~~) + \young(:~,~~,~:) \;.
\end{align*}
One can then check that in the Grothendieck group $\cG({\rm Mod}_\cH)$ we get the decomposition,
\[
[V] = [M({\tiny \young(~~)})] + [M({\tiny \young(~)})]
\]
so when $n \geq 3$ all of these modules are nonzero in that degree and we have
\[
[V_n] = [M({\tiny \young(~~)})_n] + [M({\tiny \young(~)})_n].
\]
It is easy to see that this captures all of the diagrams that can occur. We will now describe this decomposition in some detail because this is one of the few cases where we can completely work it out. In general it is very hard to determine.

The first submodule $F_1$ will be exactly the submodule spanned by the elements $x_ix_j$. It is clear that injections between degrees preserve these elements that the Demazure operators also respect these elements. One can check that the sub-$\cH$-module spanned by these elements is isomorphic to $M({\tiny \young(~~)})$. 

The quotient $V/F_1$ is spanned by $x_i^2$ in each degree. It is not hard to check explicitly that this quotient module is isomorphic to $M({\tiny \young(~)})$. 

This is a particularly nice example because $V$ is actually isomorphic to a direct sum of $M(\alpha)$, but in general this will not be the case.

\section{Cohomology of Borel Groups} \label{CohomologySection}
Before we begin our next example we will review some necessary group cohomology. Let $G$ be a group and let $M$ be a left $G$-module. Thus $M$ is a representation of $G$. Then we can describe the cohomology of $G$ with coefficients in $M$ using either homogeneous or nonhomogeneous cochains. In this section, we will only consider homogeneous cochains.

Given some nonnegative integer $q$, let $\fC^i(G,M)$ denote the group consisting of $M$-valued functions $\phi \colon G^{i+1} \to M$ satisfying
\[
\phi(g g_0,\dots, g g_i) = g \phi(g_0,\dots,g_i)
\]
for all $g,g_0,\dots,g_i \in G$. This $\fC^i(G,M)$ is called the group of {\em homogeneous cochains}. We now define a co-boundary map $\delta \colon \fC^i(G,M) \to \fC^{i+1}(G,M)$ as follows
\[
(\delta \phi)(g_0,\dots,g_{i+1}) = \sum_{j=0}^{i+1} (-1)^j \phi(g_0,\dots,g_{j-1}, \hat{g_j}, g_{j+1},\dots,g_{i+1})
\]
where $\hat{g_j}$ means we omit this element, for all $\phi \in \bC^{i+1}(G,M)$ and $(g_0,\dots,g_{i+1}) \in G^{i+2}$. Note this satisfies $\delta^2 = 0$. Now the corresponding $i$-th cohomology group of $G$ with coefficients in $M$ is given by
\[
H^i(G,M) \coloneqq K^i(G,M)/I^i(G,M)
\]
where $K^i(G,M)$ is the kernel of the map $\delta \colon \fC^i(G,M) \to \fC^{i+1}(G,M)$ and $I^i(G,M)$ is the image of the map $\delta \colon \fC^{i-1}(G,M) \to \fC^{i}(G,M)$. Note that this makes sense because $\delta^2 = 0$. 

We will now discuss how the $0$-Hecke algebra $H_n(0)$ acts on $H^i(B(n,q), \bF_q)$ where $B(n,q) \subset \GL_n(\bF_q)$ is the group of upper-triangular matrices over $\bF_q$. First we will discuss how this action works, we will then discuss how to get a $\cH$ action on the assignment $[n] \mapsto H^i(B(n,q), \bF_q)$. 

There is a much broader theory for the following, but we will restrict to the case that matters to us. For a more rigorous treatment, including proofs that everything is well defined we refer the reader to \cite{RW, Le}.

\begin{definition}
Let $G$ be a group. Two subgroups $\Gamma$ and $\Gamma'$ are said to be {\em commensurable} if 
\[
[\Gamma \colon \Gamma \cap \Gamma'] < \infty, \qquad [\Gamma' \colon \Gamma \cap \Gamma'] < \infty,
\]
that is, if $\Gamma \cap \Gamma'$ has finite index in both $\Gamma$ and $\Gamma'$.
\end{definition}

It is not hard to show that this defines an equivalence relation on subgroups of $G$. With this in mind we can define the {\em commensurator} of a subgroup $\Gamma$ in $G$ as follows
\[
\tilde{\Gamma} \coloneqq \{\alpha \in G \; | \; \alpha^{-1} \Gamma \alpha \sim \Gamma\}.
\]

Let $\Delta$ be a sub-semigroup of the group $G$ and let $C(\Delta)$ denote the set of mutually commensurable subgroups $\Gamma$ of $G$ such that
 \[
 \Gamma \subseteq \Delta \subseteq \tilde{\Gamma}.
 \]
 Given some $\Gamma \in C(\Delta)$ and $R$ a commutative ring with identity, we denote by $\cH_R(\Gamma; \Delta)$ the free $R$ module generated by the double cosets $\Gamma \alpha \Gamma$ with $\alpha \in \Delta$ and call this the {\bf Hecke algebra} associated to $\Gamma$ and $\Delta$ over $R$. So every element can be written as
 \[
 \sum_{\alpha \in \Delta} c_\alpha \Gamma \alpha \Gamma.
 \]
 where $c_\alpha \in R$ are zero except for a finite number of $\alpha$. This has a multiplication operation which is associative,
 \[
 \left(\sum_\alpha a_\alpha \Gamma \alpha \Gamma \right) \cdot \left( \sum_{\beta} b_\beta \Gamma \beta \Gamma\right) \coloneqq \sum_{\alpha,\beta} a_\alpha b_\beta (\Gamma \alpha \Gamma) \cdot (\Gamma \beta \Gamma).
 \]
 If we have right coset decompositions $\Gamma \alpha \Gamma = \bigcup \Gamma \alpha_i$ and $\Gamma \beta \Gamma = \bigcup \Gamma \beta_j$, then we can realize this product as
 \[
 (\Gamma \alpha \Gamma) (\Gamma \beta \Gamma) = \sum_{\Gamma \gamma \Gamma} m(\Gamma \alpha \Gamma, \Gamma \beta \Gamma \; ; \; \Gamma \gamma \Gamma)(\Gamma \gamma \Gamma),
 \]
 where $\alpha, \beta, \gamma$ are all in a prefixed transversal $\Omega$ and $m(\Gamma \alpha \Gamma, \Gamma \beta \Gamma \; ; \; \Gamma \gamma \Gamma) \coloneqq |\{(i,j) \; | \; \alpha_i \beta_j \in \Gamma \gamma\}$.
 
The Hecke algebra $\cH(\Gamma ; \tilde{\Gamma})$ acts on the cohomology $H^i(\Gamma, A)$ where $\Gamma$ is a subgroup of a group $G$ and $A$ is a unitary left $\bZ[G]$-module, where this is the integral group ring of $G$. If we take $G = \GL(n,q)$, $\Gamma = B_n$ the Borel subgroup and $A = \bF_q$, then $\tilde{\Gamma} = \GL(n,q)$ and $\cH(B_n ; \GL(n,q)) = \cH_n(q)$. If we restrict scalars to a field of characteristic $q$ then we recover the $0$-Hecke algebra. To see this explicitly we refer the reader to \cite{HR}[Pages 6-7]

The action is defined in the following way. Given a homogeneous cochain $\psi \colon \Gamma^{i+1} \to M$ and a double coset $\Gamma \alpha \Gamma$ with $\alpha \in \tilde{\Gamma}$, we define a $\bk$-linear map
\[
\fT(\alpha) \colon \fC^i(\Gamma,M) \to \fC^i(\Gamma,M)
\]
by
\[
(\fT(\alpha)\psi)(\gamma_0,\dots,\gamma_i) \coloneqq \sum_{j=1}^d \alpha_j^{-1} \psi(\xi_j(\gamma_0),\dots,\xi_j(\gamma_i))
\]
where $\Gamma \alpha \Gamma = \bigcup_{1 \leq j \leq d} \Gamma \alpha_j$ and $\alpha_j \gamma_k = \xi_j(\gamma_k) \alpha_{\iota(\gamma_k)}$. This gives a well defined action of $\cH(\Gamma;\tilde\Gamma)$ on the cohomology $H^i(\Gamma,A)$ for any $i$ as seen in \cite{RW, Le}.

As mentioned above $\cH(B(n,q),\GL(n,q)) \cong H_n(0)$, where $\fT(E_{i,i+1})$ is identified with the generator $\ol{\pi_i}$, $E_{i,i+1}$ being the permutation matrix swapping $i$ and $i+1$ \cite{HR}. So the above gives a well defined action of $H_n(0)$ on $H^i(B(n,q),\bF_q)$ for any choice of $i \geq 0$. The main theorem we will prove in this section is:

\begin{theorem}
For any fixed $i \geq 0$, the assignment $[n] \mapsto H^i(B(n,q),\bF_q)$ is a $\cH$-module.
\end{theorem}

Before we are ready to prove this theorem, we need a few lemmas. First we will explicitly define the embeddings from $H^i(B(n,q)) \to H^i(B(n+1,q))$. Notice there is a group homomorphism $r_{n+1,n} \colon B(n+1,q) \to B(n,q)$ given by restricting to the first $n$ rows and $n$ columns of any matrix in $B(n+1,q)$. It is not hard to check that this is indeed a group homomorphism. It is well known that such a group homomorphism induces a map on cohomology
\[
\Phi_{n} \colon H^i(B(n,q)) \to H^i(B(n+1,q))
\]
where the indexing is swapped because cohomology is contravariant. Explicitly this map on the level of cochains is defined by
\[
\Phi_n(\phi)(\gamma_0,\dots,\gamma_i) = \phi(r_{n+1,n}(\gamma_0),\dots,r_{n+1,n}(\gamma_i)),
\]
where $\phi \in \fC^i(\Gamma_n,\bF_q)$ and $\gamma_j \in B(n+1,q)$. We mention this because it will be useful later. We ultimately wish to show that these maps in combination with the above $H_n(0)$-action combine to define a $\cH$-module structure. We will prove this in stages, ultimately working towards invoking Theorem \ref{criterion}. First we must prove some structure theorems about the cosets that appear in our decomposition.

\begin{lemma} \label{cosetReps}
For a fixed permutation matrix $E_{i,i+1}$, if we consider the decomposition $\Gamma_n E_{i,i+1} \Gamma_n = \sqcup_{i} \Gamma_n \alpha_i$ there are exactly $q$ cosets independent of $n$ and we can take as coset representatives matrices $M$ with $M_{i+1,i+1} = x$, $M_{i,i+1} = M_{i+1,i} = 1$, $M_{j,j} = 1$ for $j \not= i,i+1$ and $M_{j,k} = 0$ else including $j=k=i$, where $x$ ranges over all elements of $\bF_q$.
\end{lemma}

\begin{proof}
If we were to take left cosets instead of right this coset decomposition would correspond to finding all distinct complete flags that are equivalent to the complete flag corresponding to $E_{i,i+1} \Gamma_n$ up to a left action of $\Gamma_n$. This is precisely because $\Gamma_n$ is the stabilizer of the right $\GL(n,q)$-action on complete flags of an $n$-dimensional vector space, so left cosets correspond to distinct complete flags. If we act on the left instead of the right $B(n,q)$ acts by upward row operations instead of column operations. So we can once again view cosets as complete flags if we identify the bottom row with the first spanning vector, the $(n-1)$st row with the second spanning vector etc. Then the left $B(n,q)$ action stabilizes these complete flags. 


So we must now find all complete flags equivalent to the complete flag corresponding to $\Gamma_n E_{i,i+1} $ up to a right action of $\Gamma_n$. The complete flag corresponding to $\Gamma_n E_{i,i+1}$ is exactly 
\[
V_0 \subset V_1 \subset \cdots \subset V_{n-(i+1)} \subset V_{n-i} \subset \cdots \subset V_{n-1}
\]
with $V_j = {\rm span}_{\bF_q}(e_n,\dots,e_{n-j})$ for $j=0,\dots,n-(i+2)$, $V_{n-(i+1)} = {\rm span}_{\bF_q}(e_n,\dots,e_{i+2}, e_i)$, $V_{n-i} = {\rm span}{\bF_q}(e_n,\dots,e_{i+2}, e_i, e_{i+1})$ and $V_{j} = {\rm span}_{\bF_q}(e_n,\dots,e_{n-j},e_i, e_{i+1},e_{i-1},\dots,e_{n-j})$ for $j = n-i+1,\dots,n-1$.

All of these equivalent distinct flags are exactly in correspondence with the $M_x$ described above. That is $M_{i+1,i+1} = x$, $M_{i,i+1} = M_{i+1,i} = 1$, $M_{i,i} = 0$ and all other entries are zero, where $x$ ranges over all elements of $\bF_q$. This is clear because the only difference in flags is what multiple of the vector $e_{i+1}$ we add to $e_{i}$ to get $V_{n-(i+1)}= {\rm span}_{\bF_q}(e_n,\dots,e_{i+2},e_{i}+xe_{i+1})$ and $V_{n-i} = {\rm span}_{\bF_q}(e_n,\dots,e_{i+2},e_i+ xe_{i+1}, e_{i+1})$ in the complete flag
\[
V_0 \subset \cdots V_{n-(i+1)} \subset V_{n-i} \subset \cdots \subset V_{n-1}.
\]
There are precisely $q$ such distinct choices.
\end{proof}

\begin{lemma} \label{cohomologyEquivariance}
The map $\Phi_n\colon H^i(B(n,q)) \to H^i(B(n+1,q))$ is $H_n(0)$-equivariant where $H_n(0)$ acts via the natural embedding into $H_{n+1}(0)$ on $H^i(B(n+1,q))$.
\end{lemma}

\begin{proof}
The action is defined on the level of cochains, so we can prove equivariance there and it will imply equivariance on cohomology. Given a cochain map $\phi \in \fC^i(\Gamma_n,\bF_q)$ and any $\pi_i \in H_n(0)$ acting by $\fT(E_{i,i+1})$ we wish to show $\Phi_n \fT(E_{i,i+1}) \phi = \fT(E_{i,i+1})\Phi_n \phi$. If we compute the left hand side we get,
\[
 \fT(E_{i,i+1})\Phi_n \phi (\gamma_0,\dots,\gamma_i) = \sum_{j = 1}^q M_{x_j}^{-1} \phi(r_{n+1,n}(\xi_j(\gamma_0)),\dots,r_{n,n+1}(\xi_j(\gamma_i))),
\]
where here we are using the explicit definition of the map $ \fT(E_{i,i+1})$ along with Lemma \ref{cosetReps}, so each $M_{x_j}$ is the $(n+1) \times (n+1)$ matrix described there. If we compute the right hand side,
\[
\Phi_n \fT(E_{i,i+1}) \phi =\sum_{j = 1}^q (M_{x_j})_{n+1,n+1}^{-1} \phi(\xi_j(r_{n+1,n}(\gamma_0)),\dots, \xi_j(r_{n+1,n}(\gamma_i))),
 \]
 where in this case the coset representatives we get are exactly the $n \times n$ minors of $M_{x_j}$ where we delete the last row and column, this is $(M_{x_j})_{n+1,n+1}$. This is only true because we are considering $\fT(E_{i,i+1})$ for $i = 1,\dots,n-1$.
 
First notice that each of these matrices acts trivially because we take coefficients in the ground field. To prove these two are equal, it suffices to show $\xi_j(r_{n+1,n}(\gamma_k)) = r_{n+1,n}(\xi_j(\gamma_k))$.
 
By definition $\xi_j(r_{n+1,n}(\gamma_k))$ is the matrix determined by
\begin{equation}\label{Matrix1}
 (M_{x_j})_{n+1,n+1} r_{n+1,n}(\gamma_k) = \xi_j(r_{n+1,n}(\gamma_k)) (M_{x_m})_{n+1,n+1} 
\end{equation}
and $\xi_j(\gamma_k)$ is determined by
\begin{equation} \label{Matrix2}
M_{x_j} \gamma_k = \xi_j(\gamma_k) M_{x_\ell}.
\end{equation}
In Equation \eqref{Matrix1} on the left hand side, we are multiplying minors, so one can see that this is precisely the multiplication that occurs in Equation \eqref{Matrix2} on the left hand side if you remove the last row and column of each matrix since restriction is a group homomorphism. This implies that the right hand sides must be equal as well if we restrict the right hand side of Equation \eqref{Matrix2} to only consider the $n \times n$ minors where we delete the last row and column, i.e.  $\xi_j(r_{n+1,n}(\gamma_k)) = r_{n+1,n}(\xi_j(\gamma_k))$. From the above this implies equivariance.
\end{proof}

\begin{lemma} \label{trivialAction}
If we take any element in the image of $\Phi_{n,m} = \Phi_{m} \Phi_{m-1} \cdots \Phi_{n}$, $\fT(E_{i,i+1})$ will act by $0$ on this element for $i = n+1,\dots,m$.
\end{lemma}

\begin{proof}
Once again we can work at the level of cochains. Given a cochain map $\phi \in \fC(\Gamma_n,\bF_q)$ we will show that $\fT(E_{i,i+1}) \Phi_{n,m} \phi = 0$. To do this, we can explicitly compute this map,
\[
\fT(E_{i,i+1}) \Phi_{n,m} \phi (\gamma_0,\dots,\gamma_i) = \sum_{j = 1}^q M_{x_j}^{-1} \phi(r_{m,n}(\xi_j(\gamma_0)),\dots, r_{m,n}(\xi_j(\gamma_i))),
\]
where once again $M_{x_j}$ is the $m \times m$ matrix described in Lemma \ref{cosetReps}. This detail will actually not matter for this proof, the important fact here is that there are exactly $q$ cosets. Since each $M_{x_j}^{-1}$ acts trivially as we take coefficients in the ground field, we can simplify this equation,
\[
\fT(E_{i,i+1}) \Phi_{n,m} \phi (\gamma_0,\dots,\gamma_i) = \sum_{j = 1}^q \phi(r_{m,n}(\xi_j(\gamma_0)),\dots, r_{m,n}(\xi_j(\gamma_i))).
\]
Again recall that $\xi_j(\gamma_k)$ is defined by
\[
M_{x_j} \gamma_k = \xi_j(\gamma_k) M_{x_m}
\]
where each of the $M_{x_\ell}$ have $(M_{x_\ell})_{i+1,i+1} = x_{\ell}$, $(M_{x_\ell})_{i,i+1} =(M_{x_\ell})_{i+1,i}=1$ and $M_{a,b} = 0$ else. In particular if we consider the $n \times n$ minor where we delete the last $m-n$ rows and columns we get the identity matrix since $i \geq n+1$. As a result, if we apply this restriction map $r_{m,n}$ to both sides of the above equation because it is a group homomorphism we find
\[
r_{m,n}(\gamma_k) = r_{m,n}(\xi_j(\gamma_k)),
\]
where notice this does not depend on $j$. As a result,
\[
\fT(E_{i,i+1}) \Phi_{n,m} \phi (\gamma_0,\dots,\gamma_i) = \sum_{j = 1}^q \phi(r_{m,n}(\gamma_0),\dots,r_{m,n}(\gamma_i)) = q\phi(r_{m,n}(\gamma_0),\dots,r_{m,n}(\gamma_i))  = 0 . \qedhere
\]
\end{proof}

We are now ready to prove our main theorem. 

\begin{theorem}
For any fixed $i \geq 0$, if we let $\pi_j \in H_n(0)$ act by $\fT(E_{j,j+1}) + {\rm id}$ and we take as our transition maps the $\Phi_n \colon H^i(B(n,q),\bF_q) \to H^i(B(n+1,q),\bF_q)$ then the assignment $[n] \mapsto H^i(B(n,q),\bF_q)$ is a $\cH$-module.
\end{theorem}

\begin{proof}
According to Theorem \ref{criterion} it suffices to show that the transition maps $\Phi_n$ are $H_n(0)$-equivariant and that the $\fT(E_{j,j+1}) + {\rm id}$ act via the identity on the image of $\Phi_{n,m}$ for $j = n+1,\dots,m$. The first statement is exactly what we show in Lemma \ref{cohomologyEquivariance} and the second statement follows immediately from Lemma \ref{trivialAction}.
\end{proof}


\begin{remark}
We believe this is a finitely generated $\cH$ module, but have not been able to prove it yet.
\end{remark}


\section{Homology of Borel Groups} \label{HomologySection}
In \cite{PSS} the authors study the group homology of the unipotent group and prove in particular that the assignment $[n] \mapsto H_i(U_n, \bF_q)$ is a finitely generated ${\bf OI}$-module for any fixed $i \geq 0$. Although being an ${\bf OI}$-module does provide some insight into the behavior of these homology groups, it is natural to ask if there is a more rigid categorical structure present. In this section we will discuss how this ${\bf OI}$-module structure can be extended to a $\cH$-module structure. An immediate consequence of the work in \cite{PSS} is that this $\cH$-module will also be finitely generated. This then implies we have representation stability.

Our ultimate goal is to gain a deeper understanding of the homology $H_i(B(n,q),\bF_q)$, but it is equivalent to the homology of $H_i(U(n,q),\bF_q)$ the unipotent subgroup, so we will study this. To see why this is the case notice we have a short exact sequence
\[
0 \to {\rm Diag}(n,q) \to B(n,q) \to U(n,q)  \to 0
\]
where ${\rm Diag}(n,q)$ are the diagonal matrices. From the Hochschild-Serre spectral sequence we have
\[
H_i(U(n,q), H_j({\rm Diag}(n,q);\bF_q)) \implies H_{i+j}(B(n,q),\bF_q).
\]
Since the order of ${\rm Diag}(n,q)$ is prime to $q$, the input vanishes when $j > 0$, so the spectral sequence immediately degenerates and gives
\[
H_i(U(n,q); H_0({\rm Diag}(n,q);\bF_q)) = H_i(B(n,q),\bF_q).
\]
${\rm Diag}(n,q)$ acts trivially on $\bF_q$, so the left side is $H_i(U(n,q),\bF_q)$.

 We will define a $0$-Hecke action on $U(n,q)$, this action will then induce an action on homology. We will then show this action is finitely generated via the arguments in \cite{PSS}.

Given an element $M \in U(n,q)$ we will define the action of $\pi_i \in H_n(0)$ for $i=1,\dots,n-1$. $\pi_i$ will act on $M$ by replaces the entry in columns $i+1$ with the entry in column $i$ for rows $1,\dots,i-1$, it will set $M_{i,i+1} = 0$ and it will replace the entry in row $i+1$ with the entry in row $i$ for columns $i+2,\dots,n$.

As an example consider the matrix 
\[
M = 
\begin{pmatrix}
1 &a &b &0\\
0 &1 &c &0\\
0 &0 &1 &0\\
0 &0 &0 &1
\end{pmatrix}
\]
Now if we apply $\pi_3$ we get
\[
\pi_3 M =
\begin{pmatrix}
1 &a & b &b\\
0 &1 &c &c\\
0 &0 &1 &0\\
0 &0 &0 &1
\end{pmatrix}
\]
And furthermore, if we apply $\pi_2$ we have
\[
\pi_2 \pi_3 M =
\begin{pmatrix}
1 &a &a&b\\
0 &1 & 0 &c\\
0 &0 &1 &c\\
0 &0 &0 &1
\end{pmatrix}
\]

\begin{proposition} \label{HnAction}
The action defined above is a well defined $H_n(0)$ action on $U(n,q)$.
\end{proposition}

\begin{proof}
We first check that the $\pi_i$ satisfy the necessary axioms. First we can see $\pi_i \pi_j = \pi_j \pi_i$ for $|i-j| > 1$. We may assume $i < j$. In this case the only interaction between $\pi_i$ and $\pi_j$ is in the four entries $(i,j), (i,j+1), (i+1,j), (i+1,j+1)$ where we would be adding rows $i$ and $i+1$ as well as the columns $j$ and $j+1$. However, if we labeled the entries as $(i,j) = a$, $(i,j+1) = b$, $(i+1,j) = c$ and $(i+1,j+1) = d$ then regardless of the order of composition, $(i+1,j+1) = a$ and all the other entries becomes zero.

Next we must check that $\pi_i \pi_{i+1} \pi_i = \pi_{i+1} \pi_i \pi_{i+1}$ for $i = 1,\dots, n-2$. If we fix some row $j$ with $1 \leq j \leq i-1$ the three entries affected in that row are columns $i,i+1$ and $i+2$. Let entry $(j,i) = a$, $(j,i+1) = b$ and $(j,i+2) = c$. Regardless of the order of composition we have entry $(j,i) = (j,i+1) = (j,i+2) = a$. The only other part of the matrix we need to check is entries $(k,\ell)$ with $k \leq \ell$ and $k = i,i+1,i+2$, $\ell = i,i+1,i+2$. However in this case we always have entries $(i,i+1) = (i,i+2) = (i+1,i+2) = 0$ and $(i,i) = (i+1,i+1) = 1$ regardless of the order of composition. Checking the columns is similar.

Finally we verify that $\pi_i^2 = \pi_i$ for $i = 1,\dots,n-1$. This is not hard to see from definition.

Next we need to verify that $\pi_i(AB) = \pi_i(A) \pi_i(B)$. Fix a row $1 \leq j \leq i-1$ and consider the entries in columns $i$ and $i+1$. $(AB)_{j,i} = \sum_{k=1}^n A_{j,k}B_{k,i}$ and $(AB)_{j,i+1} = \sum_{k=1}^n A_{j,k}B_{k,i+1}$. So $\pi_i(AB)_{j,i} =  \sum_{k=1}^n A_{j,k}B_{k,i}  = (\pi_i(A) \pi_i(B))_{j,i}$ and 
\[
\pi_i(AB)_{j,i+1} = \sum_{k=1}^n A_{j,k}B_{k,i} = (\pi_i(A) \pi_i(B))_{j,i+1}
\]
Checking the appropriate columns is similar.
\end{proof}

The above action then induces an action on homology by considering the action on homogeneous chains $\pi_i [g_1,\dots,g_k] = [\pi_i g_1,\dots,\pi_i g_k]$. This action is well defined by Proposition \ref{HnAction}. Furthermore, this action commutes with the differential, once again by Proposition \ref{HnAction}, so it induces an action on homology. 

Furthermore, there is a natural embedding of groups $\phi_{n,n+1} \colon U(n,q) \injects U(n+1,q)$ where we add a new row and column to the bottom of any matrix in $U(n,q)$ with a $1$ on the diagonal. This induces a map on homology $\Phi_{n} \colon H^i(U(n,q),\bF_q) \to H^i(U(n+1,q),\bF_q)$. Explicitly on the level of chains, we embed the element $[g_1, \dots, g_k] \to [\phi_{n,n+1}(g_1) | \cdots | \phi_{n,n+1}(g_k)]$.

We will now argue that the above action and $\Phi_n$ endow the homology of the unipotent group with a finitely generated $\cH$-module structure.

\begin{lemma} \label{homologyEquivariance}
The map $\Phi_n\colon H_i(B(n,q), \bF_q) \to H_i(B(n+1,q), \bF_q)$ is $H_n(0)$-equivariant where $H_n(0)$ acts via the natural embedding into $H_{n+1}(0)$ on $H_i(B(n+1,q), \bF_q)$.
\end{lemma}

\begin{proof}
We verify this on the level of chains, this will then imply the result for homology. Fix some $\pi_j \in H_n(0)$, so $j =1,\dots,n-1$ and some chain $[g_1,\dots,g_i]$ with $g_\ell \in U(n,q)$. Consider
\[
\Phi_n(\pi_j([g_1,\dots,g_i]) = \Phi_n([\pi_j(g_1),\dots,\pi_j(g_i)]) = [\phi_{n,n+1}\pi_j(g_1),\dots,\phi_{n,n+1}\pi_j(g_i)].
\]
Notice that $\phi_{n,n+1} \pi_j(g_\ell) = \pi_j \phi_{n,n+1}(g_\ell)$ because in $\phi_{n,n+1}(g_\ell)$ we only add a column of zeroes in rows $1,\dots,n$, so this will not affect the action of $\pi_j$ since $1 \leq j \leq n-1$.
\end{proof}

To show that the assignment $[n] \mapsto H_i(B(n,q),\bF_q)$ with transition maps given by $\Phi_n$ is a $\cH$-module it remains to show the following

\begin{lemma} \label{homologyTrivialAction}
If we take any element in the image of $\Phi_{n,m} = \Phi_{m} \Phi_{m-1} \cdots \Phi_{n}$, $\pi_j$ will act by the identity on this element for $j = n+1,\dots,m$.
\end{lemma}

\begin{proof}
Once again we verify this on the level of chains. Fix some $\pi_j \in H_m(0)$ for $n+1 \leq j \leq m$ and some chain $[g_1,\dots,g_i]$ for $g_\ell \in U(n,q)$. Notice that the embedded chain element $\Phi_{n,m}([g_1,\dots,g_i])$ has each $g_\ell$ with only nonzero entries on the diagonal below row $n$ and they are all equal to $1$. Hence when we apply $\pi_j$ it can only act by the identity because $n+1 \leq j \leq m$. 
\end{proof}

\begin{theorem} \label{homologyHModule}
For any fixed $i \geq 0$, if we let $\pi_j \in H_n(0)$ act as described above and we take as our transition maps the $\Phi_n \colon H_i(B(n,q),\bF_q) \to H_i(B(n+1,q),\bF_q)$ then the assignment $[n] \mapsto H_i(B(n,q),\bF_q)$ is a $\cH$-module.
\end{theorem}

\begin{proof}
According to Theorem \ref{criterion} it suffices to show that the transition maps $\Phi_n$ are $H_n(0)$-equivariant and that the $\pi_j$ act via the identity on the image of $\Phi_{n,m}$ for $j = n+1,\dots,m$. The first statement is exactly what we show in Lemma \ref{homologyEquivariance} and the second statement follows immediately from Lemma \ref{homologyTrivialAction}.
\end{proof}

In \cite{PSS} the authors study $[n] \mapsto H_i(U(n,q),\bF_q)$ as an ${\bf OI}$-module. They ultimately show that it is a finitely generated ${\bf OI}$-module. We will now show that when we restrict our $\cH$-action it agrees with their ${\bf OI}$-module action. Their finite generation result then implies that we have finite generation as a $\cH$-module and so the homology of the Unipotent subgroup in defining characteristic is representation stable.

We will first describe their action explicitly. For their purposes, they did not need an explicit description, so it is not in their paper, but they do describe one implicitly. Given an order preserving injection $\iota \colon [n] \injects [m]$ we describe the ${\bf OI}$-action on the level of chains. Suppose $\{\alpha_1,\dots,\alpha_{m-n}\} = [m] \setminus \iota([n])$, then for $g \in U(n,q)$, $\iota_\ast g$ is the matrix where $g_{i,j} = (\iota_\ast g)_{\iota(i),\iota(j)}$, $(\iota_\ast g)_{\alpha_i,\alpha_i} = 1$ and all other entries are zero.

\begin{example}
One can view this operation as inserting rows and columns into positions $\alpha_j$ with a $1$ on the diagonal and shifting the entries of the original matrix $g \in B(n,q)$ accordingly. Consider the example $\iota \colon [2] \injects [4]$ given by $\iota(1) = 1$, $\iota(2) = 3$. Let $g$ be the matrix
\[
\begin{pmatrix}
a & b\\
0 & c
\end{pmatrix}
\]
Then
\[
\iota_\ast g = 
\begin{pmatrix}
a & 0 & b & 0\\
0 & 1 & 0 & 0\\
0 & 0 & c & 0\\
0 & 0 & 0 & 1
\end{pmatrix}
\]
\end{example}

We can see this is the action the authors define in \cite{PSS} from \S 4.4. We will now argue that it agrees with out $\cH$-module action. The key is to realize that we still have an induced ${\bf OI}$-group structure given by ${\bf U}_n = U(n,q)$ and maps given by the $\cH$-module structure described above. Let ${\bf U}' = \Sigma (\bU)$. It is still the case that $U(n+1,q)$ is the semi-direct product $U_n \ltimes \bF_q^n$ and most importantly that this description is still functorial. 

To verify this, let ${\bf E}$ be the ${\bf OI}$-group given by ${\bf E}_n = \bF_q^n$ as described in \cite{PSS}. We then have homomorphisms of ${\bf OI}$-groups $i \colon {\bf U} \to {\bf U}'$ and $p \colon {\bf U}' \to {\bf U}$ with $pi = {\rm id}$ and ${\rm ker}(p) = {\bf E}$. Where $i$ is given by the natural embedding and $p$ forgets the final column. This functorial decomposition is the only place where the specific ${\bf OI}$-module structure is used in the proof of finite generation in \cite{PSS}. 

\begin{lemma} \label{OIAgree}
The restriction of the $\cH$-module structure on $[n] \mapsto H_k(U(n,q),\bF_q)$ to ${\bf OI}$ is a finitely generated ${\bf OI}$-module.
\end{lemma}

\begin{proof}
This follows formally from the proof of Theorem 1.4 in section 6 of \cite{PSS} by the remarks in the previous two paragraphs. The key step being that we still have the same functorial decomposition as described above, so Proposition 6.3 from \cite{PSS} holds.
\end{proof}
%
%
%
See the following example to illuminate this proof.

\begin{example}
Consider the order preserving injection $\iota(1) = 1$, $\iota(i) = i+1$ for $i =2,3$. This corresponds to the map $\pi_2 \pi_3 \Phi_{3,4}$. The action described in \cite{PSS} sends the matrix
\[
M = \begin{pmatrix}
1 &a &b\\
0 &1 &c\\
0 &0  &1
\end{pmatrix}
\]
to
\[
\begin{pmatrix}
1 &0 &a &b\\
0 &1 &0 &0\\
0 &0 &1 &c\\
0 &0 &0 &1
\end{pmatrix}.
\]
So the action on the final column is the embedding of the entries from the third column into positions $1$ and $3$. Comparing this to our action,
\[
\pi_3 \Phi_{3,4}(M)=
\begin{pmatrix}
1 &a &b &b\\
0 &1 &c &c\\
0 &0 &1 &0\\
0 &0 &0 &1
\end{pmatrix}.
\]
Now if we apply $\pi_2$,
\[
\pi_2 \pi_3 \Phi_{3,4}(M)=
\begin{pmatrix}
1 &a &a &b\\
0 &1 &0 &c\\
0 &0 &1 &c\\
0 &0 &0 &1
\end{pmatrix}.
\]
Chains are determined up to a left action of $U_n(q)$. So each matrix in our chain will be of this form, and we can apply the element
\[
\begin{pmatrix}
1 &0 &0 &0\\
0 &1 &-1 &0\\
0 &0 &1 &0\\
0 &0 &0 &1
\end{pmatrix}
\]
we get
\[
\begin{pmatrix}
1 &a &a &b\\
0 &1 &-1&0\\
0 &0 &1 &c\\
0 &0 &0 &1
\end{pmatrix}.
\]
Checking for other maps is similar. So it is not hard to see that the action on chains is equivalent in the final column, which is precisely the kernel of the map $p$, called ${\bf E}$ in \cite{PSS}.
\end{example}


\begin{theorem} \label{homologyHFinGen}
The $\cH$-module $[n] \mapsto H_i(U(n,q),\bF_q)$ with transition maps described in Theorem \ref{homologyHModule} is finitely generated as a $\cH$-module.
\end{theorem}

\begin{proof}
Lemma \ref{OIAgree} in combination with \cite{PSS}[Theorem 1.4] imply this $\cH$-module is finitely generated with respect to the ${\bf OI}$-substructure. As we have seen in Corollary \ref{OIFinGen}, this is equivalent to being finitely generated with respect to the full $\cH$-action.
\end{proof}

\begin{corollary} \label{RepStabilityUn}
For any fixed $i \geq 0$ sequence of $H_n(0)$-modules $\{H_i(U(n,q),\bF_q)\}_{n \geq 0}$ is representation stable.
\end{corollary}

\begin{proof}
This follows immediately from Theorems \ref{homologyHFinGen} and \ref{repStability}.
\end{proof}

\begin{corollary}
For any fixed $i \geq 0$ sequence of $H_n(0)$-modules $\{H_i(B(n,q),\bF_q)\}_{n \geq 0}$ is representation stable.
\end{corollary}

\begin{proof}
This follows formally from the introduction to this section. Since $H_i(U(n,q),\bF_q) = H_i(B(n,q),\bF_q)$, Corollary \ref{RepStabilityUn} implies this result.
\end{proof}

\section{Graded Components of Stanley Reisner Rings} \label{StanleySec}
We first recall some necessary definitions. An {\em abstract simplicial complex} $\Delta$ on a vertex set $V$ is a collection of finite subsets of $V$, called {\em faces}, such that any subset of a face is also a face. The {\em dimension} of a face $F$ is $|F|-1$, so that one vertex faces have dimension zero. The {\em dimension of a simplicial complex} is the maximum dimension of its faces. We say that a $(d-1)$-dimensional simplicial complex is {\em balanced} if there exists a coloring map $r \colon V \to [d]$ such that every face consists of vertices of distinct colors. The reason we call this coloring map $r$ is because we call the {\em rank set} of a face $F$, denoted by $r(F)$, all the colors of all of its vertices.

The {\em Stanley-Reisner ring} $\bF[\Delta]$ of a simplicial complex $\Delta$ over a field $\bF$ is
\[
\bF[\Delta] \coloneqq \bF[y_v \; | \; v \in V]/I_\Delta,
\]
where $I_\Delta \coloneqq \langle y_u y_v \; | \; u,v \in V, \; \{u,v\} \not\in \Delta\rangle$. So a monomial $y_{v_1}\cdots y_{v_d}$ is nonzero if and only if $v_1,\dots,v_d$ all belong to the same face of $\Delta$. This ideal does not equate monomials, it just makes some of them zero, so we can see that all nonzero monomials form a $\bF$-basis for $\bF[\Delta]$.

If $\Delta$ is balanced then its Stanley-Reisner ring $\bF[\Delta]$ is multigraded, that is any nonzero monomial $m = y_{v_1}\cdots y_{v_k}$ has a {\em rank multiset} $r(m)$

\subsection{Stanley Reisner Ring of the Boolean Algebra}
The Boolean algebra $\cB_n$ is the ranked poset of all subsets of $[n]$ ordered by inclusion with minimum element $\emptyset$ and maximum element $[n]$. The rank of an element is defined as the cardinality of the corresponding set, where clearly $|\emptyset| = 0$. Following the definition of the Stanley-Resiner ring above, if we take $\cB_n$ as our simplicial complex where the vertices are subsets of $[n]$ and the faces are chains of subsets, we see that $\bF[\cB_n] = \bF[y_A  \; | \; A \subset [n]]/I_\Delta$. In this case
\[
I_\Delta = \langle y_Ay_B \; | \; A,B \; \text{incomparable}\rangle,
\]
in other words $A$ is not a subset of $B$ and $B$ is not a subset of $A$. For example if we take $A = \{1,3,4\}$ and $B = \{2,4,5\}$. So the nonzero monomials are exactly given by weakly increasing chains of subsets, sometimes called {\em multichains}. This means that we have an $\bF$-basis $\{y_M\}$ indexed by multichains $M$ in $\cB_n$. This basis is multigraded by the rank multisets $r(M)$ of the multichains $M$. For example the multichain $\{2\}\subseteq \{2\} \subset \{1,2,4\} \subset [5]$ has $r(M) = \{1,1,3,5\}$. 

There is a natural way of encoding multichains, which \cite{Hu} uses to construct a $H_n(0)$ action on $\bF[\cB_n]$ which we recall now. For more details we refer the reader to \cite{Hu}. Let $M = (A_1 \subseteq A_2 \subseteq \cdots \subseteq A_k)$ be an arbitrary multichain of length $k$ in $\cB_n$, set $A_0 = \emptyset$ and $A_{k+1} = [n]$ by convention. For any such multichain $M$ of length $k$ in $\cB_n$, let $p_i(M) = \min\{j \in [k+1] \; | \; i \in A_j\}$, this records the first position where $i$ occurs in $M$. By definition it must appear in every position after this, so the collection $p(M) = (p_1(M),\dots,p_n(M))$ encodes the multichain $M$. The map $M \mapsto p(M)$ gives a bijection between the set of multichains of length $k$ in $\cB_n$ and the set $[k+1]^n$ of all words of length $n$ on the alphabet $[k+1]$ for any fixed integer $k \geq 0$. 

In section \ref{polynomialSec} we saw how $H_n(0)$ can act on the polynomial ring via Demazure operators using the generating set $\pi_i$. In this example, it is easier to use the generating set $\ol{\pi_i}$, once again we recall that $\pi_i = \ol{\pi_i} + 1$. Let $M = (A_1 \subseteq \cdots \subseteq A_k)$ be a multichain in $\cB_n$, then Huang defined
\[
\ol{\pi}_i(y_M) \coloneqq 
\begin{cases}
-y_M, &p_i(M) > p_{i+1}(M),\\
0, &p_i(M) = p_{i+1}(M),\\
s_i(y_M), &p_i(M) < p_{i+1}(M)
\end{cases}
\]
for $i = 1,\dots, n-1$. Huang shows this is a well defined action that respects the multigrading. We can use this to see how the $\pi_i$ should act,
\[
\pi_i(y_M) \coloneqq
\begin{cases}
0, &p_i(M) > p_{i+1}(M),\\
y_M, &p_i(M) = p_{i+1}(M),\\
s_i(y_M) + y_M, &p_i(M) < p_{i+1}(M)
\end{cases}
\]
We will now define a $\cH$ module where $[n] \mapsto \bF[\cB_n]$ and the inclusion map $\iota_{n,n+1}$ acts  by sending a multichain $M = (A_1 \subseteq \cdots \subseteq A_k)$ to the multichain with $A_{k+1} = [n+1]$ instead of $[n]$. The crossings act by the $\pi_i$ we just defined on multichains, not the $\ol{\pi}_i$. It is clear that the inclusion maps are $H_n(0)$-equivariant. From Theorem \ref{criterion} it remains to check that if we apply $(\iota_{n,m})_\ast$ that $\pi_{n+1},\dots,\pi_{m}$ act by $1$, or equivalently that $\ol{\pi}_{n+1},\dots,\ol{\pi}_{m}$ act by $0$. 

By definition of our embedding, elements in the image of $(\iota_{n,m})_\ast$ will be a polynomial in $y_A$ where every set $A$ is a subset of $[n]$, i.e. it will consist of monomials $y_M$ in $\cB_m$ where the multichain $M$ has $A_i \subset [n]$ and $A_{k+1} = [m]$. This implies that $p_{i}(M) = k+1 = p_{i+1}(M)$ for every $i = n+1,\dots,m-1$. As a result, $\pi_i(y_M) = y_M$ for every $i = n+1,\dots,m-1$ so Theorem \ref{criterion} implies the following:

\begin{theorem} \label{BooleanHModule}
The assignment $[n] \mapsto \bF[\cB_n]$, the Stanley-Reisner ring of the Boolean algebra, with transition maps as defined above is a $\cH$-module.
\end{theorem}

This module cannot possibly be finitely generated because it does not grow like a polynomial. However, there is a multigrading present, where we say that a multichain $M$ has multigrading $g = (r_1,\dots,r_k)$ if the multichain is of the form $A_1 \subseteq \cdots \subseteq A_k$ where $|A_i| = r_i$. The $H_n(0)$ action respects the multigrading \cite{Hu}, so we can consider the $\cH$-submodule given by restricting to homogeneous polynomials whose monomials correspond to multichains of a fixed length $k$ that correspond to a fixed composition $(\alpha_1,\dots,\alpha_k)$, that is the set sizes remain fixed as well. When we apply the transfer map $\tau \colon \bF[\cB_n] \to \bF[X]$ defined by
\[
\tau(y_M) \coloneqq \prod_{1 \leq i \leq k} \prod_{j \in A_j} x_j,
\]
for all multichains $M = (A_1 \subseteq \cdots \subseteq A_k)$ in $\cB_n$, these correspond exactly to degree $r_1(M) + \cdots + r_k(M)$ monomials. This map is not a ring homomorphism, but it restricts to an isomorphism $\tau \colon \bF[\cB_n^\ast] \cong \bF[X]$ of $H_n(0)$-modules. For more details we refer the reader to \cite[\S 3.4]{Hu}.

\begin{theorem} \label{finGenBoolean}
If we fix any multigrading $g$, the assignment $[n] \mapsto \bF[\cB_n]_g$ defines a homogeneous multigraded $\cH$-module that is finitely generated.
\end{theorem}

\begin{proof}
First, because the action as defined above respects the multigrading and so do our embedding maps $(\iota_{n,m})_\ast$ Theorem \ref{BooleanHModule} implies that this is a $\cH$-module.

To see that it is finitely generated, notice from the discussion above that if we have fixed multigrading $r(M) = (a_1,\dots,a_{k})$ the monomial corresponding to the multichain
\[
[a_1] \subseteq [a_2] \subseteq \cdots \subseteq [a_{k}],
\]
in lowest degree $a_k$ will generate all other monomials because after we embed to a higher degree we can apply permutations to get any multichain. In terms of monomials, this multichain corresponds to
\[
x_1^{r_1(M)} x_2^{r_2(M)} \cdots x_{a_k}^{r_{a_k}(M)}
\]
where $r_i(M)$ is the number of times $i$ occurs in the multichain $M$. We can also encode this monomial using its exponent vector as
\[
(r_1(M),r_2(M),\dots,r_{a_k}(M)) = (\gamma_1,\dots,\gamma_{a_k}).
\]
In a higher degree, in order for us to stay in the same multigrading, the corresponding monomial must have exponent vector corresponding to the above vector where we can insert $0$s and permute the $\gamma_i$ since this is equivalent to the multiset sizes remaining the same. The action of the $0$-Hecke algebra elements $\ol{\pi_i}$ on these exponent vectors is via sorting. If we have an exponent vector $(a_1,\dots,a_n)$, $\ol{\pi_i}$ will swap $a_i$ and $a_j$ if $a_i > a_j$, it will be zero if $a_i = a_j$ and it will act by $-1$ is $a_i < a_j$. By construction our original multichain corresponds to the exponent vector with $\gamma_1 \geq \gamma_2 \geq \cdots \geq \gamma_k$. 

Given any other exponent vector $\beta = (\beta_1,\dots,\beta_d)$ in a higher degree $d$ of the same multigrading, we can first inject $(\gamma_1,\dots,\gamma_{a_k})$ into that degree to get $(\gamma_1,\dots,\gamma_{a_k},0,\dots,0)$ where there are $d-a_k$ zeroes. Let $(\beta_{i_1},\dots,\beta_{i_{a_k}})$ be the nonzero entries of the exponent vector $\beta$. Use the $\ol{\pi_i}$ to sort the $\gamma_i$ into the order in which the nonzero $\beta_{i_j}$ appear. Algorithmically do this by looking at the first entry, if $\beta_{i_1} = \gamma_1$ move on to $\beta_{i_2}$, if not this means $\beta_{i_1} = \gamma_j < \gamma_1$. Since this is the case, we can move $\gamma_j$ to the first position as it will be smaller than everything to its left. We then consider $\beta_{i_2}$, if $\beta_{i_2} = \gamma_1$ we move on to $\beta_{i_3}$, otherwise we perform the same procedure to place the correct $\gamma_j$ into the second spot. We continue in this way until the $\gamma_i$ are in the correct order.

We can then use the $\ol{\pi_i}$ to sort the $\gamma_i$ into the nonzero entries of the exponent vector $(\beta_1,\dots,\beta_d)$, we can do this because all the other entires will be zero so we can shift any nonzero entries to the right as much as we want. This shows that we can generate any monomial with the fixed multigrading $g$, which completes the proof.
\end{proof}

\begin{corollary}
For a fixed multigrading $g$, and $n$ a sufficiently large positive integer, $F[\cB_n]_g$ is representation stable. That is, there is a finite list of compositions $\alpha_i$ paired with finitely many integers $k_j \in \bZ_{\geq -1}$ such that
\[
[F[\cB_n]_g] = \bigoplus_{i,j} c_{\alpha_i,k_j} [M(\alpha_i,k_j)]
\]
where the non-negative integer $c_i$ and $k$ are independent of $n$.
\end{corollary}

\begin{proof}
This follows directly from Theorem \ref{repStability} in combination with Theorem \ref{finGenBoolean}.
\end{proof}

Another way to state this theorem is that for any fixed multigrading $g$ and for $n$ sufficiently large, there is a finite list $\{(\alpha_i,k_i)\}$ of compositions paired with integers $k_i \in \bZ_{\geq -1}$ that completely control the simple $H_n(0)$-modules that can occur in $F[\cB_n]_g$.
%

\section{Quasisymmetric Schur Modules} \label{QuasiSchurSec}
In this section, we pursue an example of a $\cH$-module that is not a ${\bf FI}$-module. These modules arise in a natural way. Through the Frobenius characteristic map, we get an isomorphism between the irreducible representations of the symmetric group up to isomorphism and the ring of symmetric functions. In this setting, the Specht modules $V_\lambda$ map to $s_\lambda$ the Schur polynomial. 

There is a similar picture for representations of the $0$-Hecke algebra. As we have discussed, there is a commutative and non-quasisymmetric characteristic map. The quasisymmetric characteristic map provides an isomorphism between the Grothendieck group of finitely generated $H_n(0)$-modules for all $n$ and the ring of quasisymmetric functions. This map sends the irreducible module $\bC_\alpha$ to the fundamental quasisymmetric function $F_\alpha$. For many years, these were thought of as the analogue of Schur functions in ${\bf QSym}$. Although they do have a multiplication rule as we have seen and studied above, they do not naturally lift many of the well known properties of Schur functions to the ring of quasisymmetric functions (expression in terms of monomial symmetric functions, Pieri rule, etc.).

In \cite{HLMW}, the authors discovered and defined this appropriate analogue which they aptly named quasisymmetric Schur functions. For more details on why these functions are a natural refinement of Schur functions in the quasisymmetric setting we refer the reader to \cite{HLMW}. It then became a natural question to ask if there were representations of $H_n(0)$ that realize these quasisymmetric Schur functions under the quasisymmetric characteristic map. Recently, in \cite{TW}, the authors define a collection of $H_n(0)$ modules for varying $n$ and prove their image is exactly the quasisymmetric Schur functions. We will now define these modules and show how it is possible to put a $\cH$-module structure on suitable collections of them. The construction of these modules also illustrates the type of symmetry that $\cH$ preserves, namely an upward symmetry. As opposed to ${\bf FI}$ which can only act when the corresponding objects have complete symmetry.

We begin by making the necessary definitions. Given a composition $\alpha = (\alpha_1,\dots, \alpha_k)$ of $n$, we define its {\bf reverse composition diagram} which we will denote by $\alpha$ as an array of left-justified boxes with $\alpha_i$ boxes in row $i$ from the top. Notice, this is very different from the ribbon tableau representation of a composition $\alpha$. The reverse composition diagram is more akin to young diagrams. We say that a box is in position $(i,j)$ if it is $i$ rows down from the top and $j$ columns in from left to right. We are now ready to make a key definition,

\begin{definition}
Given a composition $\alpha \vDash n$, we can define a {\bf standard reverse composition tableau}, abbreviated SRCT $\tau$ of shape $\alpha$ and size $n$ to be a bijective filling $\tau \colon \alpha \to \{1,\dots,n\}$ of the cells $(i,j)$ of the reverse composition diagram $\alpha$ subject to the conditions
\begin{enumerate}
\item The entries in each row must be decreasing when read from left to right,
\item The entries in the first column must be increasing when read from top to bottom,
\item The filling must satisfy the {\em triple rule}, namely, if $i < j$ and $\tau(i,k) > \tau(j,k+1)$, then $\tau(i,k+1)$ exists and $\tau(i,k+1) > \tau(j,k+1)$.
\end{enumerate}
\end{definition}

We denote the set of all SRCTs by ${\rm SRCT(\alpha)}$ maintaining the notation in \cite{TW}. For more information on the triple rule, and the definitions and constructions we refer the reader to \cite{TW}.

\begin{example}
Let $\alpha = (2,1,4)$ be a composition of $7$, then an example of an element of ${\rm SRCT(\alpha)}$ is
\[
\young(21,3,7654) \; .
\]
A nonexample is
\[
\young(61,3,7542)
\]
because this does not satisfy the triple rule and the first column is not increasing. In particular, we see that $\tau(1,1) > \tau(3,2)$, but $\tau(1,2) = 1 < \tau(3,2) = 5$.
\end{example}

Given a SRCT $\tau$ there is a notion of a corresponding descent set 
\[
{\rm Des}(\tau) = \{i \; | \; i+1 \; \text{appears weakly right of $i$}\} \subseteq [n-1].
\]
From this we can construct a {\bf descent composition} of $\tau$, ${\rm comp}(\tau) = {\rm comp}({\rm Des}(\tau))$. The collection of standard reverse composition tableau are important because they are used to define they quasisymmetric Schur function $\cS_\alpha$. Namely,

\begin{definition}
Let $\alpha \vDash n$ be a composition. Then the {\bf quasisymmetric Schur function} $\cS_\alpha$ is defined by $S_\emptyset = 1$ and
\[
S_\alpha = \sum_{\tau \in {\rm SRCT(\alpha)}} F_{{\rm comp}(\tau)}
\]
where ${\rm comp}(\tau)$ was defined above.
\end{definition}

Now we are almost ready to define a $H_n(0)$ action on SRCTs, we just need the notion of attacking blocks.

\begin{definition}
Given $\tau \in {\rm SRCT}(\alpha)$ for a composition $\alpha \vDash n$, and any positive integer $i$ with $1 \leq i \leq n-1$ we say that $i$ and $i+1$ are {\bf attacking} if one of the following is true
\begin{enumerate}
\item $i$ and $i+1$ are in the same column of $\tau$, or
\item $i$ and $i+1$ are in adjacent columns of $\tau$, with $i+1$ positioned strictly down and to the right of $i$.
\end{enumerate}
\end{definition}

Given $\tau \in {\rm SRCT}(\alpha)$ for some $\alpha \vDash n$, and a positive integer $1 \leq i \leq n-1$, let $s_i(\tau)$ denote the filling obtained by interchanging the positions of entires $i$ and $i+1$ in $\tau$. Now we are ready to define a $0$-Hecke action on ${\rm SRCT}(\alpha)$ as follows, for $1 \leq i \leq n-1$ let
\[
\ol{\pi}_i(\tau) =
\begin{cases}
-\tau & i \not\in {\rm Des}(\tau)\\
0 & i \in {\rm Des}(\tau), \; i \; \text{and} \; i+1 \; \text{attacking}\\
s_i(\tau) & i \in {\rm Des}(\tau), \; i \; \text{and} \; i+1 \; \text{non-attacking}.
\end{cases}
\]
The action that the authors in \cite{TW} describe differs from ours by a sign, we will denote their generators by $\tilde{\pi}_i$. We choose to introduce the sign because this action is more amenable to an $\cH$-module structure and is more in line with usual $0$-Hecke actions. It is also not hard to check that this is still a well defined action and that the partial order described below is the same as the one in \cite{TW}. We will, however, do this explicitly. We omit some details that overlap with the proofs in \cite{TW}. For all of the following, $\tau$ will denote some SRCT of size $n$.

\begin{lemma} \label{squareNeg}
For $1 \leq i \leq n-1$ we have $\ol{\pi}_i^2 = -\ol{\pi}_i$
\end{lemma}

\begin{proof}
If $i \not\in {\rm Des}(\tau)$, then $\pi_i(\tau) = -\tau$. So we see that $\pi_i^2(\tau) = \tau = - \pi_i(\tau)$.

If $i \in {\rm Des}(\tau)$ then this is roughly the same proof as in \cite{TW}.
\end{proof}

\begin{lemma}
For $1 \leq i,j \leq n-1$ with $|i-j| \geq 2$, we have $\ol{\pi}_i \ol{\pi}_j = \ol{\pi}_j \ol{\pi}_i$.
\end{lemma}

\begin{proof}
If neither $i$ nor $j$ belong to ${\rm Des}(\tau)$ then $\ol{\pi}_i(\tau) = \ol{\pi}_j(\tau) = -\tau$ and so $\ol{\pi}_i \ol{\pi}_j(\tau) = \tau = \ol{\pi}_j \ol{\pi}_i(\tau)$.

Suppose that $i \in {\rm Des}(\tau)$. If $i$ and $i+1$ are attacking, then this is the same proof as in \cite{TW}. Otherwise we can assume $i$ and $i+1$ are non-attacking. This means $\ol{\pi}_i(\tau) = s_i(\tau)$. If $j \not\in {\rm Des}(\tau)$ then because $|i-j| \geq 2$ we also have $j \not\in {\rm Des}(s_i(\tau))$. As a result $\ol{\pi}_j \ol{\pi}_i(\tau) = \ol{\pi}_j(s_i(\tau)) = - s_i(\tau) = \ol{\pi}_i \ol{\pi}_j(\tau)$. If $j \in {\rm Des}(\tau)$ our generators agree with \cite{TW}. 
\end{proof}

\begin{lemma}
For $1 \leq i \leq n-2$, we have $\ol{\pi}_i \ol{\pi}_{i+1} \ol{\pi}_i = \ol{\pi}_{i+1} \ol{\pi}_i \ol{\pi}_{i+1}$.
\end{lemma}

\begin{proof}
We will proceed in cases. The first case is if $i,i+1 \not\in {\rm Des}(\tau)$. This means $\ol{\pi}_i(\tau) = \ol{\pi}_{i+1}(\tau) = - \tau$ so the desired identity clearly holds.

If $i \not\in {\rm Des}(\tau)$ but $i+1 \in {\rm Des}(\tau)$ we have $\ol{\pi}_i(\tau) = -\tau$. If $i+1$ and $i+2$ are attacking, then $\ol{\pi}_{i+1}(\tau) = 0$. This implies that $\ol{\pi}_i \ol{\pi}_{i+1} \ol{\pi}_i(\tau) = \ol{\pi}_{i+1} \ol{\pi}_i \ol{\pi}_{i+1}(\tau) = 0$. As a result we may assume that $i+1$ and $i+2$ are non-attacking, or equivalently that $\ol{\pi}_{i+1}(\tau) = s_{i+1}(\tau)$. We then have three possibilities
\begin{enumerate}
\item If $i \not\in {\rm Des}(s_{i+1}(\tau))$ then $\ol{\pi}_i\ol{\pi}_{i+1}(\tau) = - s_{i+1}(\tau)$, so $\ol{\pi}_{i+1} \ol{\pi}_i \ol{\pi}_{i+1} = -\ol{\pi}_{i+1}\ol{\pi}_{i+1}(\tau) = s_{i+1}(\tau)$ by Lemma \ref{squareNeg}. By assumption $\ol{\pi}_i(\tau) = - \tau$ so $\ol{\pi}_i \ol{\pi}_{i+1} \ol{\pi}_i(\tau) = s_{i+1}(\tau)$ as well.
\item If $i \in {\rm Des}(s_{i+1}(\tau))$ with $i$ and $i+1$ attacking in $s_{i+1}(\tau)$ then $\ol{\pi}_{i+1}\ol{\pi}_i \ol{\pi}_{i+1}(\tau) = 0 = \ol{\pi}_i \ol{\pi}_{i+1} \ol{\pi}_i(\tau)$.
\item Finally if $i \in {\rm Des}(s_{i+1}(\tau))$ with $i$ and $i+1$ non-attacking in $s_{i+1}(\tau)$ then $\ol{\pi}_i \ol{\pi}_{i+1}(\tau) = s_i s_{i+1} \tau$. As in \cite{TW} $i+1$ is not a descent in $s_i s_{i+1}$ so $\ol{\pi}_{i+1} \ol{\pi}_i \ol{\pi}_{i+1}(\tau) = - s_i s_{i+1}(\tau)$. This is precisely $\ol{\pi}_i \ol{\pi}_{i+1} \ol{\pi}_i(\tau)$ since $\ol{\pi}_i(\tau) = -\tau$.
\end{enumerate}

The next case to consider is if $i \in {\rm Des}(\tau)$ and $i+1 \not\in{\rm Des}(\tau)$. In this case we have $\ol{\pi}_{i+1}(\tau) = -\tau$. For the exact same reason as above we may assume that $i$ and $i+1$ are non-attacking otherwise both identities vanish. This means $\ol{\pi}_i\ol{\pi}_{i+1} = - s_i(\tau)$. Again we have three posibilities
\begin{enumerate}
\item If $i+1 \not\in {\rm Des}(s_i(\tau))$ then $\ol{\pi}_{i+1}(s_i(\tau)) = - s_i(\tau)$. This implies that $\ol{\pi}_{i+1} \ol{\pi}_i \ol{\pi}_{i+1} = s_i(\tau) = \ol{\pi}_i \ol{\pi}_{i+1} \ol{\pi}_i(\tau)$ since $\ol{\pi}_i \ol{\pi}_{i+1} \ol{\pi}_i(\tau) = - \ol{\pi}_i(s_i(\tau)) = - \ol{\pi}_i^2(\tau) = s_i(\tau)$.
\item If $i+1 \in {\rm Des}(s_i(\tau))$ with $i+1$, $i+2$ attacking in $s_i(\tau)$ we have $\ol{\pi}_{i+1}\ol{\pi}_i(\tau) = \ol{\pi}_{i+1}(s_i(\tau)) = 0$. This implies that both identities are zero.
\item Finally, if $i+1 \in {\rm Des}(s_i(\tau))$ with $i+1$ and $i+2$ non-attacking in $s_i(\tau)$ this means $\ol{\pi}_{i+1}\ol{\pi}_i(\tau) = s_{i+1}s_i(\tau)$. Again notice $i \not\in {\rm Des}(s_{i+1}s_i(\tau))$ so $\ol{\pi}_i \ol{\pi}_{i+1} \ol{\pi}_i(\tau) = - s_{i+1}s_i(\tau)$. Now $\ol{\pi}_{i+1} \ol{\pi}_i \ol{\pi}_{i+1}(\tau) = - s_{i+1}s_i(\tau)$ as well because $\ol{\pi}_{i+1}(\tau) = -\tau$.
\end{enumerate}

For the final case, suppose $i, i+1 \in {\rm Des}(\tau)$. In this case, the proof in \cite{TW} applies to our case because there is never a point where an element is not in a descent set. 
\end{proof}

The above proves that our definition of the $\ol{\pi_i}$ defines a $H_n(0)$ action on ${\rm SRCT}(\alpha)$ with $\alpha \vDash n$. When defining our $\cH$-module, we prefer to use the generators $\pi_i$, which we describe explicitly in this case,
\[
\pi_i(\tau) =
\begin{cases}
0 & i \not\in {\rm Des}(\tau)\\
\tau & i \in {\rm Des}(\tau), \; i \; \text{and} \; i+1 \; \text{attacking}\\
s_i(\tau) + \tau & i \in {\rm Des}(\tau), \; i \; \text{and} \; i+1 \; \text{non-attacking}.
\end{cases}
\]
In \cite{TW} the authors then place a partial order on elements of ${\rm SRCT}(\alpha)$ called $\preceq_\alpha$, where $\tau_1 \preceq_\alpha \tau_2$ if and only if there exists an element $\pi_\sigma \in H_n(0)$ such that $\pi_\sigma(\tau_1)=\tau_2$. Notice, this would not be possible for {\bf FI}. An important part of establishing that partial order is well defined, i.e. the anti-symmetry, follows from the fact that $H_n(0)$ preserves upward symmetry. 

They then arbitrarily extend this partial order to a total order $\preceq^t_\alpha$ and define new modules $\cV_{\tau_i}$ for some SRCT $\tau_i \in {\rm SRCT}(\alpha)$ with
\[
\cV_{\tau_i} = {\rm span}\{\tau_j \; | \; \tau_i \preceq^t_\alpha \tau_j\} \qquad \text{for} \; 1 \leq i \leq m.
\]
One can think of this as the poset ideal generated by $\tau_i$. They prove that for any choice of $\tau_i$, $\cV_{\tau_i}$ is a $H_n(0)$-module. Once again, we note that there is no natural $\Sigma_n$ action on these modules, or on SRCTs in general. The reason for studying these $\cV_{\tau_i}$ is that if we take $\tau$ as the minimal element of ${\rm SRCT}(\alpha)$ with respect to $\preceq_\alpha^t$ then ${\rm ch}(\cV_\tau) = \cS_\alpha$ the quasisymmetric Schur function. 

We want to note that changing $\pi_i$ does not affect $\cV_{\tau_i}$ as defined in \cite{TW}. This follows immediately from the observation that we only get a new nonzero SRCT if we fall into the third case where $i \in {\rm Des}(\tau)$ with $i$ and $i+1$ non-attacking. So suppose we have
\[
\tilde{\pi}_\sigma(\tau_1) = \tau_2
\]
then $\pi_\sigma(\tau_1) = \tau_2 + \sum_i \pi_{\sigma_i}(\tau_1)$ where $\sigma_i$ is the first $i$ generators in $\sigma$. By definition this tail will be in $\cV_{\tau_1}$ and so $\tau_2 \in \cV_{\tau_1}$. Conversely suppose we have $\pi_{\sigma}(\tau_1) = \tau_2$, then $\tilde{\pi}_\sigma(\tau_1) = \tau_2 - \sum_i \pi_{\sigma_i}(\tau_1)$ so the definition of $\cV_\tau$ does not change if we change the $H_n(0)$ action as above.

Now with our definition of $V_\tau$ we claim there is a natural $H_n(0)$ equivariant embedding $\Phi_n \colon {\rm SRCT}(\alpha) \to {\rm SRCT}(\alpha')$ for $\alpha \vDash n$ and $\alpha' \vDash n+1$ where $\alpha' = \alpha + (1)$ given by filling in $\alpha'$ with the elements of $\alpha$ and placing $n+1$ in the new square. First, notice this is $H_n(0)$ equivariant because the descent set does not change and we do not shift the elements of $\alpha$ so $\pi_1,\dots,\pi_{n-1}$ act in the same way.

Now define a $\cH$-module $\cH(\cV_\tau)$ with $\tau \vDash n$ via the assignment $[i] \mapsto \cV_{\tau + (1^{i-n})}$ for $i \geq n$ and $[i] \mapsto 0$ for $i < n$. And define the map from degree $i$ to degree $i+1$ via the map in the previous paragraph. 

\begin{theorem}
For a composition $\tau \vDash n$, the assignment $[i] \mapsto \cV_{\tau + (1^{i-n})}$ for $i \geq n$ and $[i] \mapsto 0$ for $i < n$, with the order preserving injection $\iota_{i,i+1}$ from degree $i$ to degree $i+1$ defined via $\Phi_i$ is a $\cH$-module.
\end{theorem}

\begin{proof}
By Theorem \ref{criterion} it remains to check in degree $m > n$ that $\pi_{n+1},\dots,\pi_{m-1}$ act by the identity on the image of the above map. Notice that any SRCT in the image has $n+1,\dots,m-1$ descending in the first column. This implies that they are all in the descent set and attacking, so by definition $\pi_i$ acts by the identity for $i = n+1,\dots,m-1$.
\end{proof}

\begin{example}
In practice, consider the example
\[
\young(21,3,54) \; .
\]
If we apply $\Phi_{5}$ we get
\[
\young(21,3,54,6).
\]
If we wish to embed to a higher degree, say $9$ we get
\[
\young(21,3,54,6,7,8,9).
\]
\end{example}

We can view these modules $\cH(\cV_\tau)$ as analogues of $L_\lambda^{\geq D}$ in \cite{SS4}, but in this case rather than growth in the first row we have growth in the first column. 

\begin{proposition} \label{finitelyGenSchur}
The $\cH$-module $\cH(\cV_\tau)$ is finitely generated for any choice of composition $\tau$.
\end{proposition}

\begin{proof}
It is not hard to see that the images of the embeddings from the lowest degree generate the entire $\cH$-module. This is because in any SRCT of shape $(\tau,1^d)$ we must fill the $1^d$ section with the $d$ largest natural numbers in $[|\tau|+d]$, so the number of SRCTs of shape $(\tau,1^d)$ is actually the same as the number of SRCTs of shape $\tau$.
\end{proof}

\begin{corollary}
For any $n \geq 0$ and fixed composition $\tau \vDash n$, the sequence of $H_j(0)$-modules $\{\cV_{\tau + (1^{i-n})}\}_{i \geq n}$ satisfies representation stability.
\end{corollary}

\begin{proof}
This follows immediately from Proposition \ref{finitelyGenSchur} and Theorem \ref{repStability}.
\end{proof}

\begin{remark}
It is important to remember here that we follow the notation in \cite{TW, HLMW} where they encode compositions as standard reverse composition tableau. This should not be confused with the above encoding of compositions as ribbon tableau. The above states that for any $i \gg 0$, there is a finite fixed list of compositions $\{(\alpha^j, k_j)\}$ that completely encodes all the irreducible representations that will appear in $\cV_{\tau + (1^{i-n})}$. 
\end{remark}

\section{Further Questions}
This work was largely motivated by Sections \ref{CohomologySection}, \ref{HomologySection} and the work in \cite{GS, PSS}. In particular, the sequence of modules $H^i(B(n,q),\bF_q)$ are a natural object to study, but they do not have any natural ${\bf FI}$-module structure. This is largely due to the fact that there is no apparent complete symmetry, only an upward symmetry. There are other sequences of modules with this same property such as the quasisymmetric Schur modules seen in Section \ref{QuasiSchurSec} that have received a large amount of attention recently. Heuristically one can think of such modules as the quasisymmetric functions that are not symmetric. We believe there are many more natural examples of such modules and are already investigating a few. Some immediate further questions are summarized as follows,

\begin{enumerate}
\item Is the $\cH$-module $[n] \mapsto H^i(B(n,q),\bF_q)$ studied in Section \ref{CohomologySection} finite generated as a $\cH$-module? We believe it is, but have not been able to prove it.
\item What other properties of ${\bf FI}$-module are also satisfied by $\cH$-modules? We have a notion of representation stability and finite regularity, is there a notion of depth or weight?
\item Is there any hope of bounding the regularity? We suspect this is impossible because it seems like the ability to do this depends on the Gabriel-Krull dimension being finite, which it is not for $\cH$-modules.
\item Is there a refined version of representation stability if we only consider sequences of projective $H_n(0)$-modules?
\item Is there a more concrete connection between $\cH$ and the ring of quasisymmetric functions? The Grothendieck group of ${\bf FI}$-modules is isomorphic to two copies of the ring of symmetric functions \cite{SS3}, is there some analogue for $\cG({\rm Mod}_\cH)$? We suspect there might be an infinite analogue.
\item Symmetric function theory suggests a deeper connection between ${\bf FI}$ and $\cH$. In particular, every symmetric function is quasisymmetric, and every space with complete symmetry also has partial symmetry. It then becomes natural to ask, does every ${\bf FI}$-module has a natural $\cH$-module structure? 

To make this more concrete, there is another category ${\bf FI}_q$, the $q$-deformation of ${\bf FI}$, over the ring $\bC[q]$, where the $q=1$ fiber gives ${\bf FI}$ and the $q=0$ fiber recovers ${\cH}$.  In this way, we get a correspondence between Grothendieck groups $\bK({\rm Mod}_{\bf FI}) \leftarrow \bK({\rm Mod}_{{\bf FI}_q}) \rightarrow \bK({\rm Mod}_{\cH})$. Evidence suggests the first functor, taking the fiber at $q=1$, is close to an equivalence if one considers flat ${\rm Mod}_{{\bf FI}_q}$-modules. This would then give a map in the desired direction.
\item What can one say about the extensions of $\cH$-modules?
\end{enumerate}

\bibliographystyle{amsalpha}
\bibliography{HeckeBibliography}

\end{document}